\documentclass[a4paper,10pt]{amsart}
\usepackage{amsmath,amsthm,amsfonts,amssymb,calrsfs}

\usepackage[all]{xy}
\usepackage{tikz-cd, tikz}

\usepackage{graphicx}

\usepackage{subfigure}
\usepackage{epic}
\usepackage{eepic}
\usepackage{setspace}
\usepackage{booktabs}
\usepackage{longtable}
\usepackage{pstricks}

\usepackage{yfonts}
\usepackage{xcolor}

\usepackage{array}
\usepackage{multirow}
\usepackage{caption}

\usepackage{todonotes}

\theoremstyle{plain}
\newtheorem{theorem}{Theorem}[section]
\newtheorem{lemma}[theorem]{Lemma}

\newtheorem{proposition}[theorem]{Proposition}

\theoremstyle{definition}
\newtheorem{definition}[theorem]{Definition}
\newtheorem{definition-theorem}[theorem]{Definition-Theorem}
\newtheorem{example}[theorem]{Example}
\theoremstyle{remark}
\newtheorem{remark}[theorem]{Remark}

\def\Aut{\mathrm{Aut}}

\def\GL{\mathrm{GL}}

\def\Im{\mathrm{Im}}
\def\Id{\mathrm{Id}}

\def\Lie{\mathrm{Lie}}

\def\cC{\mathcal{C}}
\def\cM{\mathcal{M}}

\def\cO{\mathcal{O}}
\def\cD{\mathcal{D}}

\def\cS{\mathcal{S}}

\def\cL{\mathcal{L}}

\def\cX{\mathcal{X}}

\def\cU{\mathcal{U}}
\def\cW{\mathcal{W}}

\def\cR{\mathcal{R}}

\def\FF{\mathbb{F}}

\def\CC{\mathbb{C}}
\def\PP{\mathbb{P}}

\def\cX{\mathcal{X}}

\def\cU{\mathcal{U}}

\def\delb{\overline\partial}

\def\K{\mathbb{K}}
\def\R{\mathbb{R}}
\def\Q{\mathbb{Q}}
\def\Z{\mathbb{Z}}
\def\N{\mathbb{N}}

\def\P{\mathbb{P}}

\def\C{\mathbb{C}}

\def\Om{\Omega}
\def\om{\omega}
\def\ep{\varepsilon}

\def\>{\rangle}

\def\<{\langle}
\def\>{\rangle}

\def\Fut{\mathrm{Fut}}

\def\Hom{\mathrm{Hom}}

\def\Spec{\mathrm{Spec}}

\def\trace{\mathrm{trace}}  

\def\dim{\mathrm{dim}}

\title[Foldable fans, cscK surfaces and local K-moduli]
{Foldable fans, cscK surfaces and local K-moduli}

\author[C. Tipler]{Carl Tipler}
\address{Univ Brest, UMR CNRS 6205, Laboratoire de Math\'ematiques de Bretagne Atlantique, France}
\email{carl.tipler@univ-brest.fr}

\date{\today}

\begin{document}
 \begin{abstract}
We study the moduli space of constant scalar curvature K\"ahler surfaces around the toric ones. To this aim, we introduce the class of foldable surfaces : smooth toric surfaces whose lattice automorphism group  contain a non trivial cyclic subgroup. We classify such surfaces and show that they all admit a constant scalar curvature K\"ahler  metric (cscK metric). We then study the moduli space of polarised cscK surfaces around a point given by a foldable surface, and show that it is locally modeled on a finite quotient of a toric affine variety with terminal singularities.
\end{abstract}

\maketitle

 
\section{Introduction}
\label{sec:intro}

The construction of moduli spaces of varieties is a central problem in complex geometry. Pioneered by Riemann, the case of moduli spaces of curves is now fairly well understood. In higher dimension however, the situation becomes much more delicate. Indeed, iterated blow-ups (see e.g. \cite[Example 4.4]{Kollar}) or the presence of non-discrete automorphisms induce non-separatedness in moduli considerations. It is then remarkable that despite their very simple combinatorial description, toric surfaces are subject to those two issues, and typically correspond to pathological points in the moduli space of surfaces. In this paper, we will show that when one restricts to constant scalar curvature K\"ahler (cscK) surfaces, the moduli space enjoys a very nice structure near its toric points.

Our motivation to restrict to cscK surfaces comes from the Yau--Tian--Donaldson conjecture (YTD conjecture, see \cite{yau, tian, donaldson}) which predicts that the existence of a constant scalar curvature K\"ahler metric (cscK metric for short) on a given polarised K\"ahler manifold should be equivalent to (a uniform version of)  K-polystability. In the Fano case, the proof of this conjecture  led to the recent construction of the moduli space of K-polystable Fano varieties by Odaka and Li--Wang--Xu \cite{Odaka2015,LiWangXu18,LiWangXu19} (see \cite{tian,Berman16,CDS,Tian15} for the YTD conjecture and also \cite[Part II]{Xu} and reference therein for a historical survey on the related moduli space). Moreover, from Odaka's work \cite{Odaka2012,Odaka2013}, in the canonically polarised case, K-stability is equivalent to having semi-log canonical singularities, a class that was successfully introduced to construct the so-called KSBA moduli space (see \cite{Kollar} for a survey on moduli of varieties of general type). While a general algebraic construction of a moduli space for K-polystable varieties seems out of reach at the moment, on the differential geometric side, Dervan and Naumann recently  built a separated coarse moduli space for compact polarised cscK manifolds \cite{DN}, extending Fujiki and Schumacher's construction that dealt with the case of discrete automorphism groups \cite{FuSchu} (see also Inoue's work on the moduli space of Fano manifolds with K\"ahler--Ricci solitons \cite{Inoue}).

However, it seems to the author that the Fano case, and its log or pair analogues, are the only sources of examples for those moduli spaces, when non--discrete automorphism groups are allowed (see also the discussion in \cite[Example 4.16]{DN} that may lead to further examples via moduli of polystable bundles). Restricting to surfaces, moduli spaces of K\"ahler--Einstein Del Pezzos were explicitly constructed in \cite{OSS}. Our focus will then be on the moduli space of polarised cscK surfaces, and its geometry around points corresponding to the toric ones. Given that they satisfy the YTD conjecture by Donaldson's work \cite{donaldson}, toric surfaces draw a lot of attention, and they appear as natural candidates to test the general machinery on  (compare also with \cite{Petra21,Petra22}). As the full classification of cscK toric surfaces is yet to be carried out, we will consider a subclass whose fans carry further symmetries, which we introduce now.
\begin{definition}
 \label{def:intro}
 Let $N$ be a rank two lattice and $\Sigma$ a complete smooth fan in $N_\R:=N\otimes_\Z \R$.
 The fan $\Sigma$ is called {\it foldable} if its lattice automorphism group $\Aut(N,\Sigma)$ contains a non-trivial cyclic subgroup. The associated toric surfaces will be called {\it foldable  surfaces}.
\end{definition}
In order to understand this class of surfaces amongst the toric ones, we show that all crystallographic groups arise as lattice automorphism groups of two dimensional smooth fans (see Section \ref{sec:latticeautom}):
\begin{proposition}
 \label{prop:intro fans}
 Let $N$ be a rank two lattice and $\Sigma$ a complete smooth fan in $N_\R$. Then  $\Aut(N,\Sigma)$ is isomorphic to one of the groups in the following set
 $$
 \lbrace C_1, C_2, C_3, C_4,C_6, D_1, D_2, D_3, D_4, D_6 \rbrace.
 $$
 Moreover, any group in the above list is isomorphic to the lattice automorphism group of some complete two dimensional smooth fan. 
\end{proposition}
In Proposition \ref{prop:intro fans}, we denote the cyclic group of order $p$ by $C_p$ and the dihedral group of order $2p$ by $D_p$. We make a distinction between $C_2$ and $D_1$ by assuming that $C_2$ acts via $-\Id$ on $N$ while $D_1$ acts through a reflection (see Section \ref{sec:latticeautom}). Then, from Proposition \ref{prop:intro fans}, a complete and smooth two dimensional fan $\Sigma$ in $N_\R$ is foldable if and only if $\Aut(N,\Sigma)$ is not isomorphic to $C_1$ or $D_1$. The associated class of toric surfaces then provides a wide class of examples of cscK surfaces (see Proposition \ref{prop:existence csck}) :

\begin{proposition}
 \label{prop:intro}
 Let $X$ be a foldable surface.   Then $X$ admits a cscK metric in some K\"ahler class.
\end{proposition}

This result relies on the classification of foldable  surfaces (Section \ref{sec:classification}) and an application of Arezzo--Pacard--Singer and Sz\'ekelyhidi's blow-up theorem for extremal K\"ahler metrics \cite{ArPacSing,szeII}. Similar arguments were used in \cite{TipTohoku} to show that any toric surface admits an iterated toric blow-up with a cscK metric. 

\begin{remark}
 To our knowledge,  all known examples of toric cscK surfaces are foldable (see e.g. \cite{TipTohoku} for a family of examples with unbounded Picard rank, and \cite{WangZhou} for a complete classification of polarised cscK toric surfaces up to Picard rank $4$). It would be interesting to produce examples of non foldable toric cscK surfaces.
\end{remark}

Our main result shows that the moduli space of polarised cscK surfaces, which is known to be a complex space \cite{FuSchu,DN}, is quite well behaved around a foldable one, as it is locally modeled on a finite quotient of a toric terminal singularity (see Section \ref{sec:singularity} for a definition of this class of singularities that originated in the MMP). 

\begin{theorem}
 \label{theo:intro}
 Let $(X,[\omega])$ be a polarised cscK foldable surface with fan $\Sigma$ in $N_\R$, and let $G=\Aut(N,\Sigma)$. Denote by $[X]\in \cM$ the corresponding point in the moduli space $\cM$ of polarised cscK surfaces. Then, there are :
 \begin{enumerate}
  \item[(i)] A Gorenstein toric affine $G$-variety $W$ with at worst terminal singularities,
\item[(ii)] Open neighborhoods $\cW$ and $\cU$ of respectively (the image of) the torus fixed point $x\in \cW \subset W/G$ and of $[X]\in \cU\subset \cM$, 
 \end{enumerate}
 such that $(\cU,[X])$ and $(\cW,x)$ are isomorphic as pointed complex spaces.
\end{theorem}

This result sheds some light on the singularities that may appear on the moduli space of polarised cscK surfaces. Note that by \cite[Theorem 4]{BraunGLM}, the good moduli space of smooth K-polystable Fano manifolds of fixed dimension and volume has klt type singularities (see Section \ref{sec:singularity} for definitions). Also, a combination of Braun, Greb, Langlois and Moraga's work on GIT quotients of klt type singularities \cite[Theorem 1]{BraunGLM}, together with the construction of Dervan and Naumann's moduli space for cscK polarised manifolds \cite[Section 3]{DN}, directly implies that the moduli space of polarised cscK surfaces has klt type singularities around its toric points (see Proposition \ref{prop:moduliaroundtoric}). Hence, Theorem \ref{theo:intro} provides a refinement of Proposition \ref{prop:moduliaroundtoric}, showing that, up to a finite covering, the toric structure is  locally preserved on the moduli space, and that the singularities are terminal, at a foldable point. On the other hand, one may not expect such a nice structure at toric boundary points of a (still hypothetical) compactification $\overline{\cM}$. Indeed, the local structure of the K-moduli space of Fano varieties was studied around singular toric varieties in \cite{Petra21,Petra22}, where much wilder phenomena were observed, mainly due to the existence of obstructed deformations.

\begin{remark}
It would be interesting to understand whether the final quotient $W/G$ in Theorem \ref{theo:intro} still is terminal or not. Note however that it is a non-trivial problem to settle when a finite quotient of a terminal singularity is terminal (see  \cite[Remark page 161]{KM}).
\end{remark}

\begin{remark}
 Due to cohomology vanishings for toric surfaces, the germ of the moduli space of polarised cscK surfaces at a toric point does not depend on the chosen cscK polarisation. See however \cite{SekTip} for local wall-crossing type phenomena when the polarisation is pushed away from the cscK locus of the K\"ahler cone.
\end{remark}

\begin{remark}
 Theorem \ref{theo:intro} provides another instance where canonical K\"ahler metrics, or K-polystability, turns useful in moduli problems.
 We would like to mention two other approaches to produce moduli spaces for (or around) toric varieties. In \cite{Mee}, a notion of analytic stack is introduced to produce moduli spaces of integrable complex structures modulo diffeomorphisms. Those spaces carry more information, as they classify a wider class of varieties, but are typically non-separated. In another direction, leaving the classical setting to the non-commutative one, quantum toric varieties admit moduli spaces which are orbifolds, see \cite{KLMV}.
\end{remark}

We proceed now to an overview of the proof of Theorem \ref{theo:intro}. If $(X,[\omega])$ is a cscK foldable surface, relying on \cite{szedef,DN} and an observation in \cite{RT}, the deformation theory of $(X,[\omega])$ is entirely encoded by the set of unobstructed polystable points in $H^1(X,TX)$ under the action of $\Aut(X)$, the latter being reductive by Matsushima and Lichnerowicz's theorem \cite{matsushima57,Lichn}. Nill's study of toric varieties with reductive automorphism groups \cite{Nill} then shows that either $X$ is isomorphic to $\C\P^2$ or $\C\P^1\times\C\P^1$ in which case it is rigid, or $\Aut^0(X)\simeq  T$, where $T$ is the torus of $X$. Then, by Ilten's work \cite{Ilten11}, $H^2(X,TX)=0$, so that locally, the moduli space we are after is given by (a neighborhood of the origin in) the GIT quotient $H^1(X,TX)//\Aut(X)$, where $\Aut(X)\simeq T\rtimes G$. Thanks again to \cite{Ilten11},
the deformation theory of toric surfaces is explicitly described in terms of their fan, and the affine toric variety $W:=H^1(X,TX)//T$ in Theorem \ref{theo:intro} can be explicitly computed. 
 The proof of Theorem \ref{theo:intro} then goes as follows. By a classification result (see Lemma \ref{lem:minimal models}), $X$ can be obtained by successive $C_p$-equivariant blow-ups of some ``minimal'' models in a list of $6$ toric surfaces. We then use the combinatorial description of toric terminal singularities (see e.g. \cite[Section 11.4]{CLS}) to prove the result for surfaces in that list. Then, using convex geometry, we show that the desired properties for the local moduli space are preserved after the $C_p$-equivariant blow-ups, and hold around $[X]$.
 
 \begin{remark}
  In Section \ref{sec:examples}, we produce an example of a toric surface $X$ such that $\Aut(X)\simeq T$ and the quotient $H^1(X,TX)//T$ is {\it not} $\Q$-Gorenstein. We do not know whether this surface carries a cscK metric. If not, the expectation is that the theory should be extended to smooth pairs $(X,D)$ where $D\subset X$ is a simple normal crossing divisor, considering singular or Poincar\'e type cscK metrics that are non-singular away from $D$, such as in \cite{AAS}.
 \end{remark}

The paper is organised as follows. In Section \ref{sec:background}, we settle the necessary background material on cscK manifolds. Then, Section \ref{sec:foldable toric surfaces} is devoted to the classification of foldable surfaces and the proofs of Proposition \ref{prop:intro fans} and Proposition \ref{prop:intro}. In Section \ref{sec:local moduli construction}, we carry over the construction of the local moduli spaces, and prove Theorem \ref{theo:intro}. Finally, in Section \ref{sec:relate moduli}, we provide maps between the local moduli spaces and speculate about their Weil--Petersson geometry, while in Section \ref{sec:higher dimension} we discuss the higher dimensional case.

\subsection*{Notations}
We will use the notations from \cite{CLS}. 
For a toric variety $X$, $\Aut(X)$ stands for its automorphism group, and $\Aut^0(X)\subset \Aut(X)$ the connected component of the identity. We denote $T$ the torus of $X$, $N$ the lattice of its one parameter subgroups, with dual lattice $M$ and pairing $\langle m , u \rangle$, for $(m,u)\in M\times N$. The fan of $X$ will be denoted $\Sigma$ (or $\Sigma_X$) and letters $\tau,\sigma$ will be used for cones in $\Sigma$. For $0\leq j\leq \dim(X)$, $\Sigma(j)$ is the set of $j$-dimensional cones in $\Sigma$. For $\K\in \lbrace \Q, \R,\C\rbrace$, we let $N_\K:= N\otimes_\Z \K$, and similarly $M_\K=M\otimes_\Z\K$. Finally, $X_\Sigma$ may be used to refer to the toric variety associated to $\Sigma$.

\subsection*{Acknowledgments} 
The author would like to thank Ruadha\'i Dervan and Cristiano Spotti for answering his questions on moduli of cscK manifolds and for several helpful comments, as well as Ronan Terpereau for stimulating discussions on the topic. The author is partially supported by the grants MARGE ANR-21-CE40-0011 and BRIDGES ANR--FAPESP ANR-21-CE40-0017, and beneficiated from the France 2030 framework porgramme, Centre Henri Lebesgue ANR-11-LABX-0020-01.

\section{Background on cscK metrics and their moduli}
 \label{sec:background}
 Let $(X,\Omega)$ be a compact K\"ahler manifold with K\"ahler class $\Omega$. We will give a very brief overview on extremal metrics and their deformations, and refer the reader to \cite{gauduchon,Szebook} for a more comprehensive treatment. 
 \subsection{Extremal K\"ahler metrics}
 \label{sec:extremal}
 Extremal K\"ahler metrics were introduced by Calabi \cite{Calabi} and provide canonical representatives of K\"ahler classes. They are defined as the critical points of the so-called Calabi functional, that assigns to each K\"ahler metric in $\Omega$ the $L^2$-norm of its scalar curvature. They include as special cases of interest cscK metrics and thus K\"ahler--Einstein metrics. The following obstruction to the existence of a cscK metric is due to Matsushima and Lichnerowicz \cite{matsushima57,Lichn}.
 \begin{theorem}[\cite{matsushima57,Lichn}]
  \label{theo:MatLich}
  Assume that $X$ carries a cscK metric. Then the automorphism group of $X$ is reductive.
 \end{theorem}
 Later on, Futaki discovered another obstruction to the existence of a cscK metric on $(X,\Omega)$ :  the vanishing of the so-called Futaki invariant \cite{Futaki83,FutBook}. Moreover, from \cite{Calabi2}, an extremal K\"ahler metric in the K\"ahler class $\Om$ is cscK precisely when the Futaki invariant $\Fut_\Om$ vanishes. Together with Arezzo--Pacard--Singer and Sz\'ekelyhidi's results on blow-ups of extremal K\"ahler metrics, we deduce the following existence result for toric surfaces :
 \begin{theorem}[\cite{ArPacSing,szeII}]
 \label{theo:equiv blow up}
  Let $(X,\Om)$ be a smooth compact polarised toric surface with cscK metric $\om\in\Om$. Let $Z\subset X$ be a finite set of torus fixed points, and  $G\subset \Aut(X)$ a finite subgroup such that :
  \begin{itemize}
  \item[(i)] The set $Z$ is $G$-invariant,
   \item[(ii)] The class $\Om$ is $G$-invariant,
   \item[(iii)] The adjoint action of $G$ on $\Lie(\Aut(X))$ has no fixed point but zero.
  \end{itemize}
Then, there is $\ep_0 >0$ such that $(\mathrm{Bl}_Z(X), \Om_\ep)$ carries a cscK metric in the class $\Om_\ep:=\pi^*\Om - \ep c_1(E)$ for $\ep\in (0,\ep_0)$, where $\pi : \mathrm{Bl}_Z(X)\to X$ stands for the blow-down map and $E=\sum_{z\in Z} E_z$ is the exceptional locus of $\pi$.
 \end{theorem}
  
\begin{proof}
From \cite{ArPacSing,szeII}, there is $\ep_0$ such that $(\mathrm{Bl}_Z(X), \Om_\ep)$ admits an extremal K\"ahler metric for $\ep\in (0,\ep_0)$. We only need to show the vanishing of the Futaki invariant $\Fut_{\Om_\ep}$. This will follow from its equivariance, as already used in \cite[Proposition 2.1]{SekTip23}. Denote $\tilde X=\mathrm{Bl}_Z(X)$. By $(i)$ and $(ii)$, $G$ lifts to a subgroup of $\Aut(\tilde X)$, such that $\Om_\ep$ is $G$-invariant. Then, by equivariance of $\Fut_{\Om_\ep} \in (\Lie ( \Aut(\tilde X)))^*$
(see e.g. \cite[Chapter 3]{FutBook} or \cite[Section 3.1]{LeSim94}), we have for all $g\in G$ and $v\in (\Lie ( \Aut(\tilde X)))$ :
$$
\Fut_{\Om_\ep}(\mathrm{Ad}_g(v))=\Fut_{\Om_\ep}(v).
$$
Let $\tilde w\in \Lie ( \Aut(\tilde X))$. Then $\tilde w$ is the lift of an element $w\in \Lie ( \Aut( X))$ that vanishes on $Z$. Moreover, 
$$
\tilde w_0 := \sum_{g\in G} \mathrm{Ad}_g(\tilde w)
$$
is a fixed point under the adjoint action of $G$ on $\Lie ( \Aut(\tilde X))$. As $Z$ is $G$-invariant, we deduce that it is the lift of a fixed point $w_0\in \Lie ( \Aut( X))$ under the adjoint action of $G$. From $(iii)$, we deduce $\tilde w_0=0$. Hence by equivariance of $\Fut_{\Om_\ep}$ :
$$
0=\Fut_{\Om_\ep}(\tilde w_0)=\sum_{g\in G} \Fut_{\Om_\ep}(\mathrm{Ad}_g(\tilde w))=\vert G\vert\; \Fut_{\Om_\ep}(\tilde w)
$$
where $\vert G \vert$ is the cardinal of $G$. This concludes the proof. 
\end{proof}

 \subsection{Local moduli of polarised cscK manifolds}
 \label{sec:local moduli cscK}
 A moduli space $\cM$ for polarised cscK manifolds has been constructed in \cite{DN}, generalising the results in \cite{FuSchu} by taking care of non-discrete automorphisms by means of \cite{szedef,Inoue}. This space is Hausdorff and endowed with a complex space structure. It is a coarse moduli space in the following sense : its points are in bijective correspondence with isomorphism classes of polarised cscK manifolds and for any complex analytic family $(\cX,\cL) \to \cS$ of polarised cscK manifolds over some reduced analytic space $\cS$ ($\cL$ stands for the relative polarisation of $\cX\to \cS$), there is an induced map $\cS \to \cM$. In what follows, we will restrict to the connected components corresponding to the moduli space of polarised cscK surfaces, still denoted by $\cM$. Roughly, $\cM$ is constructed by patching together local moduli spaces $(\cM_X)_{X\in \mathfrak{S}}$ parameterised by $\mathfrak{S}$, the set of polarised cscK surfaces, each $\cM_X$ being obtained as some open set in an analytic GIT quotient $W_X//G$. Our case of interest is when $X$ is a cscK toric surface. The first cohomology group of the tangent sheaf, namely $H^1(X,TX)$, plays a central role in our study. Indeed, it naturally arises as the space of infinitesimal deformations for $X$, and will play the role of the tangent space at $X$ to the base of a semi-universal family of deformations of $X$ (see e.g. \cite[Chapter 4]{Kodaira}). The tangent bundle $TX$ is naturally $\Aut(X)$-equivariant, meaning that the $\Aut(X)$-action lifts to an action on $TX$ in such a way that the projection map $TX\to X$ is equivariant (the lift is given by the differential of the action). Using pullbacks, we see that the bundles $\Lambda^{p,q}TX^*$ are also equivariant. Then, $\Lambda^{p,q}TX^*\otimes TX$ is equivariant, from which we deduce that $\Aut(X)$ acts naturally on $TX$-valued $(p,q)$-forms. More explicitly, for any $\Aut(X)$-equivariant bundle $E\to X$, we deduce an action of $\Aut(X)$ on the sections of $E$ by setting, for $s\in\Gamma(E)$,  $g\in\Aut(X)$ and $x\in X$:
 $$
 (g\cdot s)(x):= g( s(g^{-1}x)).
 $$
 From holomorphicity of the action, we deduce that the $\Aut(X)$-action on $TX$-valued $(0,p)$-forms  commutes with the Dolbeault operator, hence induces a representation of $\Aut(X)$ on 
 $$
 H^1(X,TX)\simeq \frac{\ker\:( \delb : \Om^{0,1}(TX)\to\Om^{0,2}(TX))}{\mathrm{Im}\:(\delb : \Om^{0,0}(TX)\to\Om^{0,1}(TX))}.
 $$
 In a more formal way, we could say that the Dolbeault resolution of the holomorphic tangent bundle is actually a resolution by injectives in the category of $\Aut(X)$-equivariant sheaves, and thus the right derived functors of the global sections functor $\Gamma$, namely the cohomology groups $H^\bullet(X,TX)$, inherit of an $\Aut(X)$-structure. We then extract from \cite{DN} the following proposition :
 
 \begin{proposition}
  \label{prop:moduliaroundtoric}
  Let $(X,\Om)$ be a polarised cscK toric surface. Then the local moduli space $\cM_X$ is given by an open neighborhood of $\pi(0)$ in the GIT quotient 
  $$
  \pi : H^1(X,TX) \to H^1(X,TX)//\Aut(X).
  $$
  In particular, its singularities are of klt type. 
 \end{proposition}
The above GIT quotient in this affine setting is given by the good categorical quotient that we now recall. Let  $R$ be the ring of regular functions on $H^1(X,TX)$. The $\Aut(X)$-action on $H^1(X,TX)$ induces an $\Aut(X)$-module structure on $R$, given by 
 $$g\cdot f:=f(g\:\cdot\:)$$ for any $g\in \Aut(X)$ and $f\in R$. We then introduce $R^{\Aut(X)}$ to be the set of $\Aut(X)$-invariant elements. By a classical result of Hilbert and Nagata, $R^{\Aut(X)}$ is finitely generated as soon as $\Aut(X)$ is reductive. In that case, the GIT quotient is given by
$$
H^1(X,TX)//\Aut(X):=\mathrm{Spec}(R^{\Aut(X)}),
$$
 with quotient map $$\pi: \mathrm{Spec}(R) \to \mathrm{Spec}(R^{\Aut(X)})$$ corresponding to the inclusion of the finitely generated algebras
 $$
 R^{\Aut(X)}\to R.
 $$
 We will then say that a point in $H^1(X,TX)$ is polystable if its $\Aut(X)$-orbit is closed, and stable if in addition its stabiliser is discrete. We refer the reader to \cite[Section 5.0]{CLS} for a short account on good categorical quotients and to \cite[Section 14.1]{CLS} for the special case of toric GIT quotients. Finally, we refer to Section \ref{sec:singularity} for the definition of klt type singularities.
 \begin{proof}[Proof of Proposition \ref{prop:moduliaroundtoric}]
 The tangent space at $X$, denoted $\tilde H^1(TX)$, of the base of a semi-universal family of complex deformations of $X$ compatible with the polarisation $\Omega$ is constructed in \cite[Lemma 6.1]{ChenSun} (see also \cite{szedef}). In general, the construction of $\tilde H^1(TX)$ goes as follows.
  First, the usual Kuranishi complex for deformations of complex structures on $X$ is 
$$
\ldots \to \Om^{0,k}(TX) \stackrel{\delb}{\to} \Om^{0,k+1}(TX) \to \ldots ,
$$
where the extension of $\delb$ to $(0,p)$-forms is given in local coordinates, for $\beta = \sum_j \beta_j \otimes \frac{\partial}{\partial z^j} $ , by 
$$
\delb \beta = \sum_j \delb \beta_j \otimes \frac{\partial}{\partial z^j}.
$$
We are interested in deformations that are compatible with some cscK metric $\om\in \Om$. Following \cite{FuSchu}, one defines maps
$$
\begin{array}{cccc}
\iota^k_\bullet  : & \Om^{0,k}(TX) & \to &\Om^{0,k+1}\\
                &  \beta & \mapsto & \iota_\beta\om
\end{array}
$$
where $\iota_\beta \om$ is obtained by the composition of first contraction and then alternation operators. For $k\geq 1$, we then set :
$$
\Om_{\om}^{0,k}(TX) := \ker \iota^k_\bullet,
$$
while we define
$$
\Om_{\om}^{0,0}(TX):= \cC^\infty(X,\C).
$$
Together with the restriction of $\delb$, and the map $\delb^0:=\cD$ defined, for $f\in\cC^\infty(X,\C)$, by
$$
\cD(f):=\delb (\nabla_{\om}^{1,0} f),
$$
where $\nabla_{\om}^{1,0} f$ is given by $\delb f = \om( \nabla_{\om}^{1,0}f , \cdot )$,  we obtain another elliptic complex $(\Om_{\om}^{0,\bullet}(TX), \delb^\bullet )$.
We denote by $\tilde H^{\bullet}(TX)$ the associated cohomology groups. Then, from this complex, following Kuranishi's techniques, one can build a semi-universal family $(\cX,\cL)\to Z$ of polarised deformations of $(X,\Omega)$, where $Z\subset B\subset \tilde H^1(TX)$ is an analytic subspace (corresponding to integrable infinitesimal deformations) of some open ball $B$ centered at zero (again, we refer to \cite[Lemma 6.1]{ChenSun} for the detailed construction). As in \cite{doan}, setting $K:=\mathrm{Aut}(X,\om)$ the group of holomorphic isometries of $(X,\omega)$, this construction can actually be made $K$-equivariantly, and even locally $K^\C$-equivariantly on $Z$, for the natural $K^\C$-action on $\tilde H^1(TX)$.
 From \cite[Section 3]{DN}, $\cM_X$ is given by an open neighborhood of $0$ in the analytic GIT quotient $W_X//G$, where $G=K^\C$, and where $W_X\subset \tilde H^1(TX)$ is the smallest $G$-invariant Stein space containing $Z$.  As noticed in \cite[Lemma 2.10]{RT}, as $X$ is toric, by Bott--Steenbrink--Danilov's vanishing of $h^{0,1}(X)$ and $h^{0,2}(X)$ (see \cite[Theorem 9.3.2]{CLS}), the vector space $\tilde H^1(TX)$ is $G$-equivariantly isomorphic to $H^1(X,TX)$. Also, for a toric surface, $\Aut(X)$ is a linear algebraic group, so that $G\simeq \Aut(X)$, as the Lie algebra of $\Aut(X)$ has no parallel vector field (see \cite[Chapter 2]{gauduchon}). Finally, $H^2(X,TX)=0$ by \cite[Corollary 1.5]{Ilten11}, so that any small element of $H^1(X,TX)$ is integrable, and $Z=B\subset H^1(X,TX)$. Hence $W_X=H^1(X,TX)$, and altogether, we obtain the first part of Proposition  \ref{prop:moduliaroundtoric}. The statement about klt type singularities follows directly from \cite[Theorem 1]{BraunGLM}, which asserts that the GIT quotient of a klt type singularity is of klt type, and which implies in particular that the GIT quotient of $H^1(X,TX)$ (which is smooth) by $\Aut(X)$ is of klt type.
 \end{proof}

 \section{Foldable toric surfaces}
\label{sec:foldable toric surfaces}
 Let $N$ be a rank two lattice and $\Sigma$ be a  fan in $N_\R$, with associated toric surface $X$. We will assume $\Sigma$ to be {\it smooth}, that is each cone $\sigma\in\Sigma$ is generated by elements in $N$ that form part of a $\Z$-basis, and to be {\it complete},  i.e. 
 $$
 \bigcup_{\sigma\in\Sigma} \sigma =N_\R.
 $$
 Hence, $X$ is a smooth and compact toric surface.

 \subsection{Lattice automorphisms and fans}
 \label{sec:latticeautom}
 We will be interested in automorphism groups of smooth toric surfaces. Denote by $T:=N\otimes_\Z\C^*$ the torus of $X$, and by $\Aut(N,\Sigma)$ the group of {\it lattice automorphisms} of $\Sigma$. Recall that $\Aut(N,\Sigma)$ is the subgroup of $\GL_\Z(N)$ consisting in elements $g\in \GL_\Z(N)$ such that the induced isomorphisms $g\in \GL_\R(N_\R)$ send $\Sigma$ bijectively to $\Sigma$. By a result of Demazure \cite[Proposition 11]{Demazure}, the automorphism group $\Aut(X)$ of $X$ is a linear algebraic group isomorphic to 
 $$
 \Aut^0(X)\rtimes G,
 $$
 where $G$ is a quotient of $\Aut(N,\Sigma)$.  We will later on be interested in the GIT quotient of $H^1(X,TX)$ by $\Aut(X)$, assuming $X$ to admit a cscK metric. A combination of  Matsushima and Lichnerowicz's obstruction together with Nill's work on toric varieties with reductive automorphism group \cite{Nill} implies :
 \begin{proposition}
  \label{prop:reductive aut group toric surface}
  Assume that $X$ is a toric surface that admits a cscK metric. Denote by $\Sigma$ its fan and $N$ its lattice of one-parameter subgroups. Then either 
  $$
  X\in\lbrace \P^2, \P^1\times\P^1 \rbrace,
  $$
  or 
  $$
 \Aut(X)\simeq T\rtimes \Aut(N,\Sigma).
 $$
 \end{proposition}
\begin{proof}
 By Matsushima and Lichnerowicz theorem \cite{matsushima57,Lichn}, $\Aut(X)$ is reductive. The result then follows from Nill's result \cite[Theorem 1.8]{Nill} together with Demazure's structure theorem \cite[Proposition 11, p. 581]{Demazure}, that we now recall. Demazure's root system for $(N,\Sigma)$, that we denote here by $\cR$ (it corresponds to $\mathrm{Rac}(\Sigma)$ in Demazure's notations \cite{Demazure}), is the set 
 $$
 \cR:=\lbrace m\in M\:\vert \:  \exists \rho\in\Sigma(1)\:\mathrm{with}\: \langle m,  u_\rho\rangle =1 \:\mathrm{but}\:\forall \rho'\in\Sigma(1),\: \rho'\neq\rho, \langle m , u_{\rho'}\rangle\leq 0\rbrace ,
 $$
 where we use the notation $u_\rho\in N$ for the primitive generator of the ray $\rho$. Demazure's structure theorem then asserts that the identity component $\Aut(X)^0$ of $\Aut(X)$ is generated by $T$ and some unipotent elements $\lbrace U_m, m\in\cR\rbrace$. In particular,
 $$
 \dim(\Aut(X))=\dim(T)+\vert \cR\vert.
 $$
 Then, the quotient $\Aut(X)/\Aut(X)^0$ is finite and isomorphic to a quotient of $\Aut(N,\Sigma)$ by a subgroup $W(N,\Sigma)\subset \Aut(N,\Sigma)$ that is  the Weyl group of a maximal reductive subgroup of $\Aut(X)$ with root system $\cR\cap-\cR$. Then, Nill's theorem \cite[Theorem 1.8]{Nill} gives an upper bound for the dimension of the automorphism group of a smooth complete toric variety, when it is reductive. In the case of a smooth toric surface $X$ with reductive automorphism group, the outcome is that either $X$ is a product of projective spaces, or $\dim(\Aut(X))=2$. In the latter case, from the previous discussion on Demazure's result, this implies that the set of Demazure's roots $\cR$ for $(N,\Sigma)$ is empty. Hence, $\Aut(X)^0=T$ and $\Aut(X)/\Aut(X)^0\simeq \Aut(N,\Sigma)$, which concludes the proof.
\end{proof}
As $\P^2$ and $\P^1\times\P^1$ are rigid, we will focus now on the case 
$$
\Aut(X)\simeq T\rtimes \Aut(N,\Sigma),
$$
and characterise the possible finite groups that arise as lattice automorphism groups of rank two complete fans.

\begin{lemma}
\label{lem:classification orders}
 Let $g\in \Aut(N,\Sigma)$, and denote by $m\in \N$ the order of $g$. Then $$m\in \lbrace 1, 2, 3, 4, 6\rbrace.$$
 Moreover, the complex linear extension of $g$ to $N_\C$ is conjugated (in $\GL(N_\C)\simeq \GL_2(\C)$) to the following  :
 \begin{enumerate}
  \item If $m=1$, then $g \sim \Id$,
  \item If $m=2$ and $\det(g)= 1$, then $g\sim -\Id$,
  \item If $m=2$ and $\det(g)=-1$, then $g\sim \left[                      
                                \begin{array}{cc}
                                      1 & 0 \\
                                      0 & -1
                              \end{array}
                              \right]$,
  \item If $m=3$, then $g\sim      \left[                      
                                \begin{array}{cc}
                                      j & 0 \\
                                      0 & j^2
                              \end{array}
                              \right]$,                       
\item If $m=4$, then $g\sim      \left[                      
                                \begin{array}{cc}
                                      i & 0 \\
                                      0 & -i
                              \end{array}
                              \right]$, 
\item If $m=6$, then $g\sim      \left[                      
                                \begin{array}{cc}
                                      -j & 0 \\
                                      0 & -j^2
                              \end{array}
                              \right]$,
 \end{enumerate}
where we denoted by $j=e^{i\frac{2\pi}{3}}$ and by $\sim$ the equivalence relation given by conjugation, and where we used an isomorphism $N_\C \simeq \C^2$.
\end{lemma}
The proof is an elementary exercise in linear algebra. We include it for convenience of the reader.
\begin{proof}
Fix and isomorphism $N\simeq \Z^2$ and identify $g$ with an element of $\GL_2(\Z)$. The characteristic polynomial of $$g =\left[                      
                                \begin{array}{cc}
                                      a & b \\
                                      c & d
                              \end{array}
                              \right]$$                    
then reads 
$$
\chi(Y)=Y^2-(a+d)Y+(ad-bc).
$$
On the other hand, Jordan's normal form for the $\C$-linear extension of $g$ implies that it is conjugated (in $\mathrm{M}_2(\C)$) to 
$$
\left[
\begin{array}{cc}
                                      \alpha & \delta \\
                                      0 & \beta
                              \end{array}
                              \right]
$$
for $(\alpha,\beta)\in\C^2$ and $\delta\in\lbrace 0,1 \rbrace$. As $g$ is both invertible and of finite order, we deduce that $\delta =0$ and $g$ is conjugated to 
$$
\left[
\begin{array}{cc}
                                      \alpha & 0 \\
                                      0 & \beta
                              \end{array}
                              \right].
$$
Moreover, $\alpha^m=\beta^m=1$, so that $\alpha$ and $\beta$ are $m$-th roots of unity in $\C$. On the other hand, by conjugation invariance, we also have 
$$
\det(g)=\alpha\beta=ad-bc\in \lbrace -1,+1 \rbrace$$ as $g\in \GL_2(\Z)$ 
and 
$$
\trace(g)=\alpha +\beta =a+d \in\Z.
$$
If $\alpha$ and $\beta$ belong to $\lbrace -1, +1 \rbrace$, that is if $m=2$, then we are done. If not, at least one of $\alpha$ or $\beta$ is not real, and thus $\alpha$ and $\beta$ must be complex conjugated roots of $\chi(Y)$. Hence 
$$\trace(g)=\alpha +\overline{\alpha}\in\Z,$$
but also 
$$\vert \alpha +\overline{\alpha} \vert <2$$
so that 
$$\alpha +\overline{\alpha} \in \lbrace -1,0, 1 \rbrace.
$$
Hence 
$$\alpha \in \lbrace \pm i, \pm j, \pm j^2 \rbrace,$$
from which the result follows easily.
\end{proof}
Recall that we denote by $C_p$ the cyclic group of order $p$ and by $D_p$ the dihedral group of order $p$.
We then have :
\begin{proposition}
 The group $\Aut(N,\Sigma)$ is isomorphic to one of the groups in the following set 
 $$
  \lbrace C_1, C_2, C_3, C_4,C_6, D_1, D_2, D_3, D_4, D_6 \rbrace.
 $$
\end{proposition}
\begin{proof}
As $\Aut(N,\Sigma)$ is finite, we can fix an $\Aut(N,\Sigma)$-invariant euclidean metric on $N_\R$ (again considering the $\R$-linear extension of $\Aut(N,\Sigma)$). Then, by the classification of finite subgroups of the group of orthogonal transformations of the plane, we deduce that $\Aut(N,\Sigma)$ is isomorphic to $C_p$ or $D_p$, for some $p\in\N^*$. As those groups admit elements of order $p$, by Lemma \ref{lem:classification orders}, the result follows. 
\end{proof}
\begin{remark}
 This proposition simply recovers the well known classification of crystallographic groups in dimension $2$.
\end{remark}

We will now show that all the above listed groups arise as lattice automorphism groups of rank two complete smooth fans. Note that $C_2$ and $D_1$ will be distinguished by the fact that their generator is $-\Id$ in the first case and a reflection in the second case.
Consider then the following fans in $\Z^2$, where we denote by $\lbrace e_1, e_2 \rbrace$ the standard basis of $\Z^2$.
\begin{example}
\label{example:hirzebruch}
Let $\FF_n$ be the $n$-th Hirzebruch surface, that is the total space of the fibration 
$$\PP(\mathcal{O}_{\P^1}\oplus \mathcal{O}_{\P^1}(n)) \rightarrow \CC\PP^1.$$
Note that $\FF_0=\CC\PP^1\times\CC\PP^1$. Then, up to isomorphism, the fan $\Sigma_{\FF_n}$ of $\FF_n$ is described by 
$$\Sigma_{\FF_n}(1)=\lbrace \R_+\cdot(0,-1),\:\R_+\cdot(1,0),\:\R_+\cdot(0,1),\:\R_+\cdot(-1,-n) \rbrace.$$
As an example, $\FF_2$
admits the fan description of Figure 1, that we will denote by $\Sigma_1'$. In this figure, and the one that follow, the dots represent the lattice points, and we only represent the ray generators of the complete two dimensional fan. 

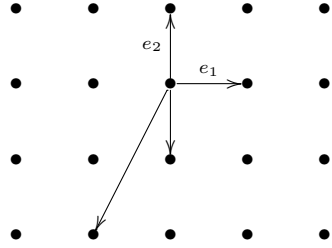
\begin{figure}[htbp]
\label{fig1prime}
\psset{unit=0.95cm}
\begin{pspicture}(-3.5,-3.5)(3.5,0.5)
$$
\xymatrix @M=0mm{
\bullet & \bullet & \bullet & \bullet & \bullet \\
 \bullet & \bullet & \bullet\ar[u]^{e_2}\ar[r]^{e_1}\ar[ldd]\ar[d]& \bullet & \bullet \\
\bullet & \bullet & \bullet & \bullet & \bullet \\
\bullet & \bullet & \bullet & \bullet & \bullet 
}
$$
\end{pspicture}
\caption{Fan $\Sigma_1'$ of $\FF_2$}
\end{figure}
\end{example}

\begin{example}
\label{example:projectivespace}
In Figure 2 and Figure 3 are depicted the fans $\Sigma_{\P^2}=\Sigma_3'$ and $\Sigma_{\P^1\times\P^1}=\Sigma_4'$ of $\P^2$ and $\P^1\times\P^1$ respectively (the numbering of the fans is motivated by Proposition \ref{prop:all groups arise} below).
\begin{figure}[htbp]
\label{fig3prime}
\psset{unit=0.95cm}
\begin{pspicture}(-3.5,-2.5)(3.5,0.5)
$$
\xymatrix @M=0mm{
\bullet & \bullet & \bullet & \bullet & \bullet \\
 \bullet & \bullet & \bullet\ar[u]^{e_2}\ar[r]^{e_1}\ar[dl]& \bullet & \bullet \\
\bullet & \bullet & \bullet & \bullet & \bullet
}
$$
\end{pspicture}
\caption{Fan $\Sigma_3'$ of $\P_2$}
\end{figure}
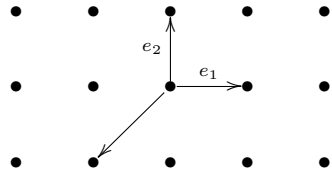

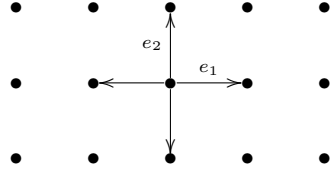
\begin{figure}[htbp]
\label{fig4prime}
\psset{unit=0.95cm}
\begin{pspicture}(-3.5,-2.5)(3.5,0.5)
$$
\xymatrix @M=0mm{
\bullet & \bullet & \bullet & \bullet & \bullet \\
 \bullet & \bullet & \bullet\ar[u]^{e_2}\ar[r]^{e_1}\ar[l]\ar[d]& \bullet & \bullet \\
\bullet & \bullet & \bullet & \bullet & \bullet 
}
$$
\end{pspicture}

\caption{Fan $\Sigma_4'$ of $\CC\PP^1\times\CC\PP^1$}
\end{figure}
\end{example}
\newpage
\begin{example}
 \label{example:blowups}

Recall that blowing-up a smooth toric surface along a torus fixed point produces a new smooth toric surface. By the orbit cone correspondence, such a fixed point corresponds to a two dimensional cone $$\sigma =\R_+\cdot e_i + \R_+\cdot e_{i+1}$$
in the fan of the blown-up surface, and the fan of the resulting surface is obtained by adding the ray generated by $e_i + e_{i+1}$ to the set of rays of the initial surface (see  \cite[Chapter 3, Section 3]{CLS}). By iterated blow-ups, we obtain the following fans. Figure 4 represents an iterated blow-up of $\P^1\times\P^1$.

\begin{figure}[htbp]
\label{fig2prime}
\psset{unit=0.95cm}
\begin{pspicture}(-3.5,-2.5)(3.5,0.5)
$$
\xymatrix @M=0mm{
\bullet & \bullet & \bullet & \bullet & \bullet \\
 \bullet & \bullet & \bullet\ar[u]\ar[r]\ar[l]\ar[d]\ar[ru]\ar[ld]\ar[urr]
\ar[lld]
& \bullet & \bullet \\
\bullet & \bullet & \bullet & \bullet & \bullet 
}
$$
\end{pspicture}
\caption{Fan $\Sigma_2'$ : iterated blow-up of $\P^1\times\P^1$}
\end{figure}
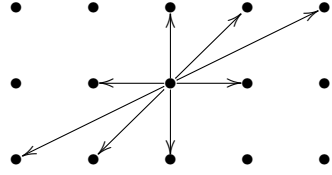

The following Figure 5 is a two points blow-up of $\P^1\times\P^1$ and at the same time a three points blow-up of $\C\P^2$. 

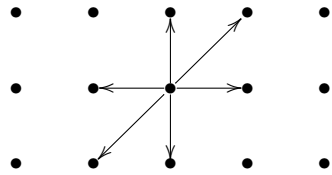
\begin{figure}[htbp]
\label{figX}
\psset{unit=0.95cm}
\begin{pspicture}(-3.5,-2.5)(3.5,0.5)
$$
\xymatrix @M=0mm{
\bullet & \bullet & \bullet & \bullet & \bullet \\
 \bullet & \bullet & \bullet\ar[u]\ar[r]\ar[l]\ar[d]\ar[ru]\ar[ld]& \bullet & \bullet \\
\bullet & \bullet & \bullet & \bullet & \bullet 
}
$$
\end{pspicture}
\caption{Fan $\Sigma_6'$ : blow-up of $\P^2$ along its three fixed points.}
\end{figure}

Here are further examples to complete our classification of lattice automorphism groups of rank two complete smooth fans (Figures 6 to 10). The first one is a single blow-up of $\FF_2$, that cancels the symmetry of $\Sigma_1'=\Sigma_{\FF_2}$.

\begin{figure}[htbp]
\label{fig5}
\psset{unit=0.95cm}
\begin{pspicture}(-3.5,-4.5)(3.5,0.5)
$$
\xymatrix @M=0mm{
\bullet & \bullet & \bullet & \bullet & \bullet \\
 \bullet & \bullet & \bullet\ar[u]^{e_2}\ar[r]^{e_1}\ar[ldd]\ar[d]\ar[lddd] & \bullet & \bullet \\
\bullet & \bullet & \bullet & \bullet & \bullet \\
\bullet & \bullet & \bullet & \bullet & \bullet \\
\bullet & \bullet & \bullet & \bullet & \bullet 
}
$$
\end{pspicture}
\caption{Fan $\Sigma_1$ : one-point blow-up of $\FF_2$}
\end{figure}
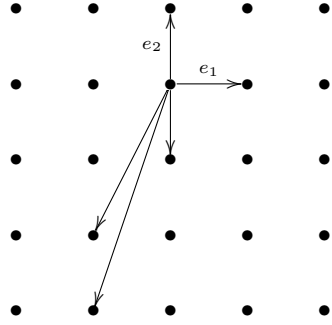
\newpage
The next one is obtained from $\Sigma_2'$ by blowing-up in a $C_2$-equivariant way (recall that $C_2=\langle -\Id \rangle$).
\begin{figure}[htbp]
\label{fig2}
\psset{unit=0.95cm}
\begin{pspicture}(-3.5,-4.5)(3.5,0.5)
$$
\xymatrix @M=0mm{
\bullet &\bullet & \bullet & \bullet & \bullet & \bullet &\bullet \\
\bullet &\bullet & \bullet & \bullet & \bullet & \bullet &\bullet \\
 \bullet &\bullet & \bullet & \bullet\ar[u]\ar[r]\ar[l]\ar[d]\ar[ru]\ar[ld]\ar[urr]\ar[uurrr]
\ar[lld]\ar[llldd]
& \bullet & \bullet &\bullet \\
\bullet &\bullet & \bullet & \bullet & \bullet & \bullet &\bullet \\
\bullet &\bullet & \bullet & \bullet & \bullet & \bullet &\bullet 
}
$$
\end{pspicture}
\caption{Fan $\Sigma_2$.}
\end{figure}
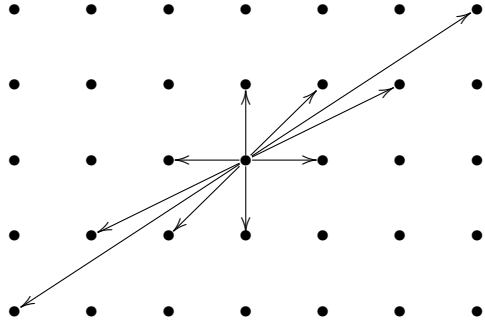

The third one is obtained from the fan of $\P^2$ by three successive blow-ups of three distinct fixed points, in a symmetric way.

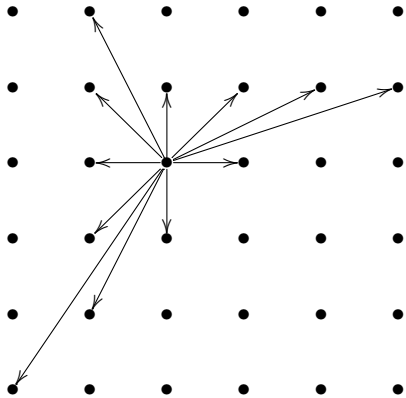
\begin{figure}[htbp]
\label{fig3}
\psset{unit=0.95cm}
\begin{pspicture}(-3.5,-5.5)(3.5,0.5)
$$
\xymatrix @M=0mm{
\bullet & \bullet & \bullet & \bullet & \bullet& \bullet \\
\bullet & \bullet & \bullet & \bullet & \bullet& \bullet \\
 \bullet & \bullet & \bullet\ar[u]\ar[r] \ar[ru]\ar[rru]\ar[rrru]\ar[l]\ar[lu]\ar[luu]\ar[llddd]\ar[ldd]\ar[ld]\ar[d] & \bullet & \bullet & \bullet\\
\bullet & \bullet & \bullet & \bullet & \bullet & \bullet\\
\bullet & \bullet & \bullet & \bullet & \bullet & \bullet\\
\bullet & \bullet & \bullet & \bullet & \bullet & \bullet
}
$$
\end{pspicture}
\caption{Fan $\Sigma_3$}
\end{figure}
The next one is obtained by two successive blow-ups of $\P^1\times\P^1$ at four fixed points.

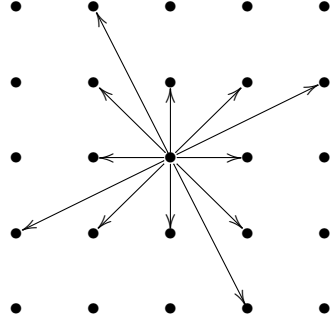
\begin{figure}[htbp]
\label{fig4}
\psset{unit=0.95cm}
\begin{pspicture}(-3.5,-4.5)(3.5,0)
$$
\xymatrix @M=0mm{
\bullet & \bullet & \bullet & \bullet & \bullet\\
\bullet & \bullet & \bullet & \bullet & \bullet \\
 \bullet & \bullet & \bullet\ar[u]\ar[r] \ar[ru]\ar[rru]\ar[l]\ar[lu]\ar[luu]\ar[lld]\ar[ld]\ar[d]\ar[rd]\ar[rdd] & \bullet &  \bullet\\
\bullet & \bullet & \bullet & \bullet & \bullet \\
 \bullet & \bullet & \bullet & \bullet  & \bullet
}
$$
\end{pspicture}
\caption{Fan $\Sigma_4$}
\end{figure}
\newpage
The last one is obtained from $\Sigma_6'$ by two successive blow-ups at six fixed points.

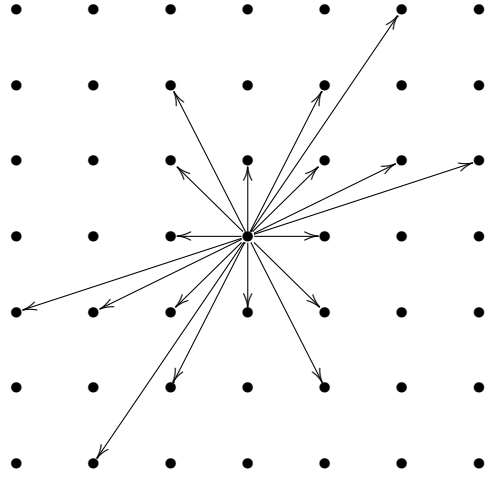
\begin{figure}[htbp]
\label{fig6}
\psset{unit=0.95cm}
\begin{pspicture}(-3.5,-6.5)(3.5,0.5)
$$
\xymatrix @M=0mm{
\bullet & \bullet & \bullet & \bullet & \bullet& \bullet& \bullet \\
\bullet & \bullet & \bullet & \bullet & \bullet& \bullet& \bullet \\
\bullet & \bullet & \bullet & \bullet & \bullet& \bullet& \bullet \\
 \bullet& \bullet & \bullet & \bullet\ar[u]\ar[r]\ar[l]\ar[d]\ar[ru]\ar[ld]\ar[rru]\ar[rrru]\ar[ruu]\ar[rruuu]\ar[luu]\ar[lu]\ar[lld]\ar[llld]\ar[ddl]\ar[dddll]\ar[rd]\ar[rdd]& \bullet & \bullet & \bullet\\
\bullet & \bullet & \bullet & \bullet & \bullet & \bullet& \bullet\\
\bullet & \bullet & \bullet & \bullet & \bullet & \bullet& \bullet\\
\bullet & \bullet & \bullet & \bullet & \bullet & \bullet& \bullet
}
$$
\end{pspicture}
\caption{Fan $\Sigma_6$.}
\end{figure}

\end{example}

\begin{proposition}
 \label{prop:all groups arise}
 Let $p\in \lbrace 1,2,3,4,6 \rbrace$. Then we have $$\Aut(\Z^2,\Sigma_p)\simeq C_p$$ and $$\Aut(\Z^2,\Sigma_p')\simeq D_p.$$
\end{proposition}
\begin{proof}
 The proof is straightforward, although a bit tedious, so we will only give a sketch of it. Fix an orientation of $\R^2$, and pick a two dimensional cone $\sigma$ of $\Sigma_p$ (the proof goes the same for $\Sigma_p'$). Let $g\in \Aut(\Z^2,\Sigma_p)$. Then $g\cdot \sigma\in\Sigma_p(2)$. Moreover, $g$ maps the two facets of $\sigma$ to the two facets of $g\cdot \sigma$. As the fan $\Sigma_p$ is explicitly described, for each possible $\tau=g\cdot\sigma\in\Sigma_p(2)$, we deduce the possible matrix coefficients for $g$ :
 $$
 g=\left[                      
                                \begin{array}{cc}
                                      a_\tau & b_\tau \\
                                      c_\tau & d_\tau
                              \end{array}
                              \right]\in\GL_2(\Z).
 $$
 For any possible $\tau$, one can then explicitly compute the images of the rays of $\Sigma_p$ by $g$. They should all belong to $\Sigma_p(1)$, and this enables to select the allowed $\tau$'s for $g$ to be a lattice automorphism of $\Sigma_p$. The result then follows easily.
\end{proof}

\begin{remark}
\label{rem:auto fans hirzebruch}
 One can also show similarly that for any Hirzebruch surface $\FF_n$ with $n\geq 1$, and with fan $\Sigma_{\FF_n}$, one has $\Aut(N,\Sigma_{\FF_n})\simeq D_1$.
\end{remark}

 \subsection{Classification of foldable surfaces}
 \label{sec:classification}
 We first recall the definition :
 \begin{definition}
  \label{def:foldable surface}
 We will say that a smooth and complete two dimensional fan is {\it foldable} if its lattice automorphism group contains a non-trivial cyclic group. A {\it foldable surface} is a toric surface whose fan is foldable.
 \end{definition}
We now proceed to a partial classification of foldable toric surfaces. Notice that if $G\subset \Aut(N,\Sigma)$, then we can find a $G$-action on $X$, faithfull as soon as $\Aut(X)\simeq T\rtimes \Aut(N,\Sigma)
$.
 \begin{definition}
 A $G$-equivariant blow-up of $X$ is a blow-up of $X$ along a $G$-orbit of torus fixed points. We will say that a smooth toric surface $\tilde X$ is obtained from $X$ by successive $G$-equivariant blow-ups if there is a sequence of $G$-equivariant blow-ups : $$\tilde X= X_k \to X_{k-1} \to \ldots \to X_1 \to X_0=X.$$
 \end{definition}
Denote by $\mathrm{Bl}_{p_1,p_2,p_3}(\P^2)$ the toric surface associated to the fan of Figure $5$, that is the blow-up of $\P^2$ at its three torus fixed points.
 \begin{proposition}
  \label{prop:classification}
  Let $X$ be a foldable surface with fan $\Sigma$. Assume that $\Aut(N,\Sigma)$ contains a subgroup isomorphic to $C_p$, $p \geq 2$.
  \begin{itemize}
   \item If $p\in \lbrace 2, 4 \rbrace$, then $X$ is obtained from $\P^1\times \P^1$ by successive $C_p$-equivariant blow-ups,
   \item If $p=3$, then $X$ is obtained from $\P^2$ by successive $C_p$-equivariant blow-ups,
   \item If $p=6$, then $X$ is obtained from $\mathrm{Bl}_{p_1,p_2,p_3}(\P^2)$ by successive $C_p$-equivariant blow-ups.
  \end{itemize}
 \end{proposition}
\begin{proof}
 Through its representation 
 $$C_p \to \GL_\R(N_\R),
 $$
 the subgroup $C_p$ of $\Aut(N,\Sigma)$ acts on $N_\R$ by rotations (once an invariant euclidean metric is fixed), hence has no fixed points. We deduce that the action of $C_p$ on $\Sigma(1)$ is free, and then $p$ divides $\vert \Sigma(1)\vert$, the number of rays of $\Sigma$.
 
 On the other hand, from the classification of toric surfaces (see e.g. \cite[Chapter 10]{CLS}), we must have $ \vert\Sigma(1)\vert\geq 3$. 
Moreover, if $\vert\Sigma(1)\vert=3$, then $X\simeq \P^2$, while if $\vert\Sigma(1)\vert=4$, then $X\simeq \FF_n$ for some $n\in \N$. By Remark \ref{rem:auto fans hirzebruch}, we see that $X\simeq \P^1\times \P^1$ in that case. 
 
 So we may assume that $\vert \Sigma(1)\vert \geq 5$, and that $p$ divides $\vert \Sigma(1)\vert$.  By the classification of toric surfaces, we know that $X$ is obtained by successive blow-ups from $\FF_n$, $n\geq 0$, or from $\P^2$. Then, there is a ray $\rho_i$ in $\Sigma(1)$ generated by the sum of two primitive elements 
 $$u_i=u_{i-1}+u_{i+1}$$
 generating adjacent rays $\rho_{i-1}$ and $\rho_{i+1}$ in $\Sigma(1)$ (that is both $\rho_i+\rho_{i-1}$ and $\rho_i+\rho_{i+1}$ belong to $\sigma(2)$). Note that
 $$
 \det(u_{i+1},u_{i-1})=\det(u_i-u_{i-1},u_{i-1})=\det(u_i,u_{i-1})=1
 $$
 so the contraction of the divisor associated to the ray $\rho_i$ to a point is smooth. By linearity, for any $g\in C_p$ we have 
 $$g\cdot u_i=g\cdot u_{i-1}+g\cdot u_{i+1}, $$
 hence we can contract to points the $C_p$-orbit of the torus invariant divisor $D_{\rho_i}$ associated to $\rho_i$ and obtain a smooth toric surface whose fan still admits $C_p$ as a subgroup of its lattice symmetry group. By an inductive argument on $\vert \Sigma(1)\vert$, we obtain the result.
\end{proof}
From the classification, we obtain the following proposition.
\begin{proposition}
 \label{prop:existence csck}
 Let $X$ be a foldable surface. Then $X$ admits a cscK metric.
\end{proposition}

\begin{proof}
 This is simply an induction on the number of $C_p$-equivariant blow-ups in Proposition \ref{prop:classification}, using Theorem \ref{theo:equiv blow up} at each step, and the fact that $\P^1\times\P^1$, $\P^2$ and $\mathrm{Bl}_{p_1,p_2,p_3}(\P^2)$ admit cscK metrics in $C_p$-equivariant classes (for $\mathrm{Bl}_{p_1,p_2,p_3}(\P^2)$, it follows again from Theorem \ref{theo:equiv blow up} applied to $\P^2$).
\end{proof}

 \section{The local cscK moduli around a foldable  surface}
 \label{sec:local moduli construction}
  Let $X$ be a smooth and complete toric surface with fan $\Sigma$ and one parameter subgroups lattice $N$. We will denote by $(u_i)_{0\leq i \leq r} $ the primitive elements of $N$ that generate the rays 
  $$\rho_i=\R_+\cdot u_i \in \Sigma(1),$$ 
  labeled in the counterclockwise order, so that for each $i$, $(u_i, u_{i+1})$ is a positively oriented $\Z$-basis of $N$, and $\rho_i+\rho_{i+1}\in\Sigma(2)$, with the convention $u_0=u_r$. 
 \subsection{Deformation theory of toric surfaces}
 \label{sec:deformations of toric surfaces}
 We recall here some results from \cite{Ilten11} regarding the deformation theory of toric surfaces. As $X$ is toric, dualizing the generalized Euler sequence yields (see e.g. \cite[Theorem 8.1.6]{CLS}) :
 $$
 0\to \mathrm{Pic}(X)^*\otimes_\Z\cO_X\to \bigoplus_{i=1}^r\cO_X(D_{\rho_i})\to TX\to 0,
 $$
 where $\mathrm{Pic}(X)$ stands for the Picard group of $X$ and $D_{\rho_i}$ is the torus invariant divisor associated to the ray $\rho_i\in\Sigma(1)$. This sequence is actually a short exact sequence of $\Aut(X)$-equivariant bundles, and induces a long exact sequence of $\Aut(X)$-modules in cohomology. From a classical vanishing theorem (see e.g.  \cite[Theorem 9.3.2]{CLS}), we have $H^1(X,\cO_X)=0$ and $H^2(X,\cO_X)=0$, and this long exact sequence induces an isomorphism of $\Aut(X)$-modules :
  \begin{equation}
   \label{eq:isom ilten}
 \bigoplus_{i=1}^r H^1(X,\cO_X(D_{\rho_i}))\simeq H^1(X,TX).
  \end{equation}
 In particular, the action of the torus $T$ on $X$ induces a representation of $T$ on 
 $$
 V:=H^1(X,TX)
 $$
 compatible with the above isomorphism. The complete reducibility theorem for torus actions (see \cite[Proposition 1.1.2]{CLS} and references therein) provides a weight space decomposition
 that we will denote by
 $$
 V=\bigoplus_{m\in M} V_m.
 $$
 More precisely, $V$ splits into $T$-invariant subspaces $V_m\subset V$ where the $T$-action on $V_m$ is given by
 $$
 t\cdot x= \chi^m(t)\,x
$$
for $t\in T$, $x\in V_m$, and where $\chi^m : T \to \C^*$ is the character associated to the weight $m\in M$. Using the isomorphism (\ref{eq:isom ilten}), Ilten obtained for each weight $m\in M=\Hom_\Z(N,\Z)$ (\cite[Corollary 1.5]{Ilten11}) :
 \begin{equation}
  \label{eq:iltenformula}
 \dim(V_m)=\sharp \displaystyle\left\{ \rho_i\in\Sigma(1)\: \left\vert \begin{array}{c}
             \langle m, u_i \rangle = -1 \\
             \langle m, u_{ i\pm 1} \rangle <0
           \end{array}
 \right.  \right\}.
 \end{equation}
 where $\sharp$ stands for the cardinal.
 \begin{example}
 \label{example:deformation spaces}
 We will provide explicit examples of weight space decompositions for $V=H^1(X,TX)$, when $X$ is the smooth toric surface associated to a foldable fan. First, $\P^1 \times \P^1$, $\P^2$ and  $\mathrm{Bl}_{p_1,p_2,p_3}(\P^2)$ are rigid, as can be checked directly using Formula (\ref{eq:iltenformula}). In those cases, $V=0$. Then, denote by $Y_2$ the foldable toric surface associated to the fan $\Sigma_2'$ of Figure 4.
 We also introduce (see Figures 11 and 12) the foldable toric surfaces $Y_4$ and $Y_3$ associated to the following fans :
 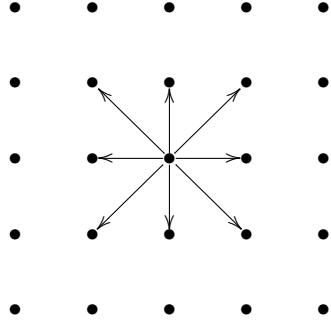
\begin{figure}[htbp]
\label{figY4}
\psset{unit=0.95cm}
\begin{pspicture}(-3.5,-4.5)(3.5,0.5)
$$
\xymatrix @M=0mm{
\bullet & \bullet & \bullet & \bullet & \bullet \\
\bullet & \bullet & \bullet & \bullet & \bullet \\
 \bullet & \bullet & \bullet\ar[u]\ar[r]\ar[l]\ar[d]\ar[ru]\ar[ld]\ar[lu]\ar[rd]& \bullet & \bullet \\
\bullet & \bullet & \bullet & \bullet & \bullet \\
\bullet & \bullet & \bullet & \bullet & \bullet 
}
$$
\end{pspicture}
\caption{Fan of $Y_4$.}
\end{figure}
 
\begin{figure}[htbp]
\label{figY3}
\psset{unit=0.95cm}
\begin{pspicture}(-3.5,-5.5)(3.5,0)
$$
\xymatrix @M=0mm{
\bullet & \bullet & \bullet & \bullet & \bullet& \bullet \\
\bullet & \bullet & \bullet & \bullet & \bullet& \bullet \\
 \bullet & \bullet & \bullet\ar[u]\ar[r] \ar[ru]\ar[rru]\ar[l]\ar[lu]\ar[ldd]\ar[ld]\ar[d] & \bullet & \bullet & \bullet\\
\bullet & \bullet & \bullet & \bullet & \bullet & \bullet\\
\bullet & \bullet & \bullet & \bullet & \bullet & \bullet\\
\bullet & \bullet & \bullet & \bullet & \bullet & \bullet
}
$$
\end{pspicture}
\caption{Fan of $Y_3$}
\end{figure}
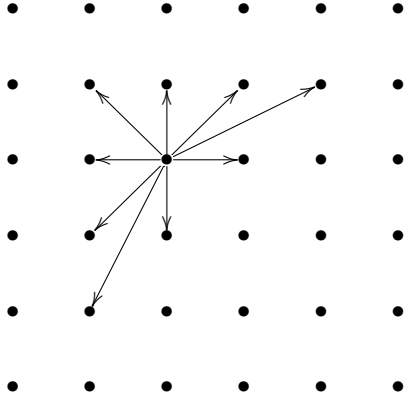
One can compute directly that for $i\in\lbrace 2,3,4\rbrace$,
$$\Aut(N,\Sigma_{Y_i})\simeq D_i.$$
  Set $V_i=H^1(Y_i,TY_i)$. Denote by $(e_1,e_2)$ the standard basis of $N=\Z^2$ and by $(e_1^*,e_2^*)$ the dual basis of $M=(\Z^2)^*\simeq \Z^2$. Testing the conditions in Formula (\ref{eq:iltenformula}) for each ray generator gives 
 \begin{equation}
  \label{eq:example Vs}
 \begin{array}{ccc}
 V_2 & = & V^1_{+e_2^*}\oplus V^1_{-e_2^*}\oplus V^1_{-e_1^*+e_2^*}\oplus V^1_{e_1^*-e_2^*}\\
  & &  \\
V_3 & = & V^2_{-e_2^*}\oplus V^2_{e_1^*}\oplus V^2_{-e_1^*+e_2^*} \\
 & & \\
 V_4 & =& V^1_{+e_1^*}\oplus V^1_{-e_1^*}\oplus V^1_{+e_2^*}\oplus V^1_{-e_2^*}\\
  & & 
 \end{array}
 \end{equation}
 where $V^j_m$ stands for a $j$-dimensional weight $m$ representation of $T$. 
 \begin{remark}
 \label{rem:Gaction}
  One can check by unraveling the isomorphisms used in \cite[Section 1]{Ilten11} that the $C_3$-action on $Y_3$ induces an action on the coordinates $\chi_j^m$ of $V_3$, with two orbits given by 
  $$\lbrace \chi_j^{-e_2^*},\chi_j^{e_1^*}, \chi_j^{-e_1^*+e_2^*}\rbrace, \: j\in\lbrace 1,2\rbrace$$
  where $\chi_j^m$, for $1\leq j\leq \dim(V_m)$, stand for generators of the weight $m$ space $V_m$. Similarly, the $D_2$ (resp. $D_4$) action on $Y_2$ (resp. $Y_4$) induces a transitive action on the coordinates of $V_2$ (resp. $V_4$).
 \end{remark}

 \end{example}

 \subsection{The toric GIT quotient}
 \label{sec:toric quotient}
 We will assume from now on that $X$ carries a cscK metric, so by Proposition \ref{prop:reductive aut group toric surface}, we have 
 $$
 \Aut(X)\simeq T \rtimes \Aut(N,\Sigma).
 $$
 We will set 
 $$
 G:= \Aut(N,\Sigma).
 $$
 From Proposition \ref{prop:moduliaroundtoric}, the local moduli space $\cM_X$ of polarised cscK surfaces around $[X]\in\cM$ is given by a neighborhood of the origin in the GIT quotient of $V$ by $T\rtimes G$. Assume now that $X$ is foldable. To prove Theorem \ref{theo:intro}, it is then enough to show that the GIT (or categorical) quotient
 $$
 W:=V//T
 $$
is toric, Gorenstein and terminal. We will end this section by showing that $W$ is indeed an affine toric variety. We first need the following lemma, where by $C_p\subset G$ we mean that there is an injection of $C_p$ in $G$.
\begin{lemma}
 \label{lem:minimal models}
 Let $X$ be a foldable toric surface. Then one of the following holds :
 \begin{itemize}
  \item  If $X$ is rigid, then $X\in\lbrace \P^2, \P^1\times \P^1 , \mathrm{Bl}_{p_1,p_2,p_3}(\P^2)\rbrace$.
  \item If $X$ is not rigid and $C_2\subset G$, $X\in\lbrace Y_2, Y_4 \rbrace$ or is obtained from $Y_2$ or $Y_4$ by successive $C_2$-equivariant blow-ups.
  \item If $X$ is not rigid and $C_3\subset G$, $X=Y_3$ or is obtained from $Y_3$ by successive $C_3$-equivariant blow-ups.
 \end{itemize}
\end{lemma}
Recall that $Y_2$, $Y_3$ and $Y_4$ were defined in Example \ref{example:deformation spaces}. 
\begin{proof}
 Note that by Example \ref{example:deformation spaces},  $\P^2$, $\P^1\times \P^1$ and $\mathrm{Bl}_{p_1,p_2,p_3}(\P^2)$ are rigid. If $X$ is not one of those three foldable surfaces, then  from Proposition \ref{prop:classification}, either it belongs to $\lbrace Y_2, Y_3, Y_4 \rbrace$, or is obtained from $Y_j$ by successive $C_p$-equivariant blow-ups, where $p=2$ for $j\in \lbrace 2,4 \rbrace$ and $p=3$ for $j=3$. From \cite[Corollary 1.6]{Ilten11}, $H^1(Y_j,TY_j)$ injects in $H^1(X,TX)$. Hence from the computations of Example \ref{example:deformation spaces}, we see that $X$ is not rigid, and the result follows.
\end{proof}

\begin{proposition}
 \label{prop:Wtoric}
 The GIT quotient $W$ inherits the structure of a normal toric affine variety.
\end{proposition}
Although the proof is straightforward, we should warn the reader that our setting is slightly different from the standard quotient construction of toric varieties (see e.g. \cite[Chapter 5]{CLS}), as the weight spaces $V_m$ may have dimension greater than $1$, and the quotient map $\tilde \sigma \to \sigma$ (see the proof below) may send several rays to the same one. This detailed proof will also settle the necessary notations for the following section.
\begin{proof}
If $V=0$, then there is nothing to prove. We then assume that $X$ is not rigid. From Lemma \ref{lem:minimal models}, $X$ is obtained from $Y_j$ by successive $C_p$-equivariant blow-ups, where $j\in\lbrace 2, 3, 4\rbrace$, for suitable $p\in\lbrace 2,3 \rbrace$.

 Denote by $d=\dim(V)$ and by $\tilde N\simeq \Z^d$ the lattice of one parameter subgroups of $\tilde T \simeq(\C^*)^d$ acting on $V\simeq \C^d$ by multiplication on each coordinate :
 $$
 \forall (t,x)\in \tilde T \times V,\: t\cdot x =(t_i x_i)_{1\leq i \leq d}.
 $$
 The $T$-module structure of $V$ is then equivalent to the injection $T \to \tilde T$, or the lattice monomorphism
 $$
 B : N \to \tilde N
 $$
 defined by 
 $$B(u)=(\langle m_i , u \rangle)_{1\leq i \leq d}$$
 where the $m_i$'s run through all the weights in the weight space decomposition of $V$ (with multiplicities when $\dim(V_m)\geq 2$). 
 The fact that $B$ is injective comes again from the injection $V_j \subset V$ (cf \cite[Corollary 1.6]{Ilten11}) and the explicit computation of $V_j$ in Example \ref{example:deformation spaces}. We then consider the quotient
 $$
 N':=\tilde N / N,
 $$
 where we identify $N$ with $B(N)$ by abuse of notation. We claim that $N'$ is again a lattice, that is $N$ is saturated in $\tilde N$. If $X=Y_j$, for $j\in \lbrace 2,3,4\rbrace$, then this can be checked directly using the description of $V_j$ in Example \ref{example:deformation spaces}. If $X$ is a blow-up of $Y_j$, then $V_j \subset V$, so that $B$ can be written $B=(B_j, B_+)$, where
 $$B_j : N \to \tilde N_j$$
 is the injection corresponding to the weight space decomposition of $V_j$, with 
 $$\tilde N_j\simeq \Z^{d_j},$$
 for $d_j=\dim(V_j)$. Then, the fact that $B_j(N)$ is saturated in $\tilde N_j\subset \tilde N$ implies that $B(N)$ is saturated in $\tilde N$. Thus, we have a short exact sequence of lattices 
 \begin{equation}
  \label{eq:short sequence lattice}
  0\longrightarrow N \stackrel{B}{\longrightarrow}\tilde N \stackrel{A}{\longrightarrow} N' \longrightarrow 0
 \end{equation}
where 
$$A : \tilde N \to \tilde N /N$$
denote the quotient map. 

Introduce $(\tilde e_i)_{1\leq i\leq d}$ the basis of $V$ dual to the coordinates $\chi^{m_i}$ of the weight spaces $V_{m_i}$. It corresponds to a $\Z$-basis of $\tilde N$, still denoted $(\tilde e_i)_{1\leq i \leq d}$.  The cone 
$$\tilde \sigma =\sum_{i=1}^d \R_+\cdot \tilde e_i\subset \tilde N_\R$$
satisfies 
$$
\Spec(\C[\tilde \sigma^\vee\cap\tilde M])=V,
$$
where
$$\tilde M=\Hom_\Z(\tilde N, \Z).
$$
Consider the cone of $N'_\R$ :
$$
\sigma = A(\tilde \sigma)= \sum_{i=1}^d \R_+\cdot A(\tilde e_i)
$$
(we will still denote by $A$ and $B$ their $\R$-linear extensions). Our goal is then to show that $W$ is isomorphic to the affine toric variety defined by $\sigma$. We first need to prove that $\sigma$ has the required properties to define such a variety. The cone $\sigma$ clearly is rational, convex and polyhedral. We now prove that $\sigma$ is strictly convex, that is 
$$-\sigma \cap \sigma =\lbrace 0 \rbrace.$$ So let $v\in-\sigma\cap \sigma$. Then there is $(u_+, u_-)\in(\tilde \sigma)^2$ such that $A (u_\pm)=\pm v$ . We deduce that 
$$u_+ + u_-\in \ker (A)\cap \tilde \sigma = \Im (B)\cap \tilde \sigma .$$
We will then show that 
\begin{equation}
 \label{eq:ImB cap sigma}
 \Im (B) \cap \tilde \sigma= \lbrace 0 \rbrace,
\end{equation}
which implies $u_+=u_-=0$ by strict convexity of $\tilde \sigma$. Notice that it is enough to show (\ref{eq:ImB cap sigma}) for $X=Y_j$. Indeed, using the decomposition $B=(B_j,B_+)$ as before, if 
$$u=B(x)\in\tilde \sigma,$$
then $$u_j=B_j(x)\in \tilde \sigma_j,$$ where $\tilde \sigma_j \subset (\tilde N_j)_\R$ is the cone corresponding to $V_j$. If 
\begin{equation}
 \label{eq:Im Bj}
\Im(B_j)\cap \tilde \sigma_j =\lbrace 0 \rbrace,
\end{equation} 
by injectivity of $B_j$, we deduce $x=0$ and then $u=0$. Remains to prove (\ref{eq:Im Bj}), which follows again from the explicit descrition of the weights of the $T$ action on $V_j$. For example, if $j=2$, and if 
$$B_2(x_1,x_2)=(x_2,-x_2,-x_1+x_2,x_1-x_2) \in \tilde \sigma_2,$$
then by definition of $\tilde \sigma_2$ we deduce 
$$
\left\{ \begin{array}{ccc}
         x_2 & \geq & 0 \\
         - x_2 & \geq & 0 \\
         -x_1 +x_2 & \geq & 0 \\
         x_1 - x_2 & \geq & 0
        \end{array}
\right.
$$
and thus $x_1=x_2=0$. The cases $j\in\lbrace 3, 4 \rbrace$ are similar. We just proved that $\sigma$ is a strongly convex rational polyhedral cone, and thus defines an affine toric variety.

We now claim that 
$$
W\simeq\Spec(\C[\sigma^\vee\cap M'])
$$
with 
$$M'=\Hom_\Z(N', \Z).
$$
It is equivalent to show
$$
\C[\sigma^\vee\cap M']\simeq \C[\chi^{m_1},\ldots,\chi^{m_d}]^T.
$$
The latter equality follows easily from the definitions. Indeed, considering the dual sequence to (\ref{eq:short sequence lattice}),
\begin{equation}
  \label{eq:short sequence lattice dual}
  0\longrightarrow M' \stackrel{A^*}{\longrightarrow}\tilde M \stackrel{B^*}{\longrightarrow} M \longrightarrow 0
 \end{equation}
for $m'\in M'$, if $\tilde m =A^* (m')$, then
$$
\begin{array}{ccc}
 \chi^{m'} \in \sigma^\vee & \Longleftrightarrow & \forall i\in [\![ 1 , d ]\!],\: \langle m', A(\tilde e_i) \rangle \geq 0\\
 & & \\
  & \Longleftrightarrow &  \forall i\in [\![ 1 , d ]\!],\: \langle A^*( m'),  \tilde e_i \rangle \geq 0\\
  & & \\
  & \Longleftrightarrow & \left\{ \begin{array}{c}
  B^*(\tilde m) =0 \\
  \forall i\in [\![ 1 , d ]\!],\: \langle \tilde m,  \tilde e_i \rangle \geq 0\end{array} \right. \\
  & & \\
   & \Longleftrightarrow &  \chi^{\tilde m}\in   \C[\chi^{m_1},\ldots,\chi^{m_d}]^T,
\end{array}
$$
where in the last equivalence we used the facts that from the definition of the basis $(\tilde e_i)_{1\leq i\leq d}$ :
$$
\chi^{\tilde m}\in   \C[\chi^{m_1},\ldots,\chi^{m_d}] \Longleftrightarrow \forall i\in [\![ 1 , d ]\!],\: \langle \tilde m,  \tilde e_i \rangle \geq 0,
$$
and  that the $T$-module structure on $\C[\chi^{m_1},\ldots,\chi^{m_d}]$ reads, for any one-parameter subgroup $\lambda_u:\C^*\to T$ generated by some $u\in N$ :
$$
\lambda_u(t)\cdot \chi^{\tilde m}=t^{\langle \tilde m, B(u)\rangle}\;\chi^{\tilde m}.
$$
This proves the claim, and the result follows.
\end{proof}
 
 \subsection{Singularities}
 \label{sec:singularity}
 We keep the notations from the previous section (see in particular the proof of Proposition \ref{prop:Wtoric}).
 Our goal here is to conclude the proof of Theorem \ref{theo:intro}, by showing that $W$ is Gorenstein and terminal. Let us first recall the definitions of those notions, and their combinatorial characterisation in the toric case (we refer the reader to \cite[Introduction and Chapter 5]{KM} for a more general treatment of those notions).
 \begin{definition}
  Let $Z$ be a normal toric variety\footnote{Being normal is not required in the most general definition of Gorenstein singularities. However, we will only deal with normal toric varieties, that are rational \cite[Theorem 11.4.2]{CLS}, hence Cohen--Macaulay, which is usually required to define Gorenstein singularities.}. Then $Z$ is {\it Gorenstein} if the canonical divisor $K_Z$ is Cartier ({\it $\Q$-Gorenstein} if $K_Z$ is $\Q$-Cartier). In that case, $Z$ has {\it terminal singularities} if there is a resolution of singularities $\pi : \tilde Z \to Z$ such that if we set
  $$ 
  K_{\tilde Z}=\pi^*K_Z + \sum_i a_i E_i,
  $$
  where the $E_i$'s are distinct irreducible divisors, then for all $i$, we have $a_i >0$.
 \end{definition}
One can check that the above definition doesn't depend on the choice of resolution. Terminal singularities play an important role in the minimal model program, being the singularities of minimal models. Their logarithmic version turned out very useful as well. Recall that a {\it log pair} is a normal variety $Z$ together with an effective $\Q$-divisor $D$ with coefficients in $[0,1]\cap \Q$. A {\it log resolution} for a log pair $(Z,D)$ is a resolution of singularities $\pi: \tilde Z \to Z$ such that the exceptional locus $\mathrm{Exc}(\pi)$ of $\pi$ is a divisor and such that  $\pi^{-1}(\mathrm{Supp}(D))\cup\mathrm{Exc}(\pi)$
is a simple normal crossing divisor.
\begin{definition}
 Let $(Z,D)$ be a log pair such that $K_Z+D$ is $\Q$-Cartier (the pair is then called {\it $\Q$-Gorenstein}). In that case, $(Z,D)$ has {\it klt singularities} if there is a log-resolution  $\pi : \tilde Z \to Z$  such that if we set
  $$ 
  K_{\tilde Z}=\pi^*(K_Z+D) + \sum_i a_i E_i,
  $$
  where the $E_i$'s are distinct irreducible divisors, then for all $i$, we have $a_i >-1$. Following \cite{BraunGLM}, we will say that a normal variety $Z$ is of {\it klt type} if there exists an effective $\Q$-divisor $D$ such that $(Z,D)$ is a $\Q$-Gorenstein log pair with klt singularities.
\end{definition}
At that stage, from \cite[Theorem 1]{BraunGLM}, we know that $W$ is of klt type. Also, by \cite[Corollary 11.4.25 ]{CLS}, if $W$ is Gorenstein, then it has log terminal singularities, meaning that the pair $(W,0)$ is klt. We will give a direct proof that it is actually Gorenstein with terminal singularities, using the following characterisation from \cite[Proposition 11.4.12]{CLS} (where we use the notations $\sigma(1)$ to denote the set of rays in a cone $\sigma$ and $u_\rho$ for the primitive generator of a ray $\rho$ in $\sigma(1)$).
\begin{proposition}
 \label{prop:Gorenstein and terminal}
 Consider $W=\Spec(\C[\sigma^\vee\cap M'])$ the affine toric variety associated to the rational strictly convex polyhedral cone $\sigma\subset N'_\R$ (with lattice $N'$). Then 
 \begin{itemize}
  \item $W$ is Gorenstein if and only if there exists $m\in M'$ such that for all $\rho\in\sigma(1)$, $\langle m, u_\rho \rangle =1$.
  \item In that case, $W$ has terminal singularities if and only if the only lattice points in 
  $$
  \Pi_\sigma:= \left\{ \sum_{\rho\in \sigma(1)} \lambda_\rho u_\rho\:\vert\: \sum_{\rho\in\sigma(1)} \lambda_\rho \leq 1,\: 0\leq \lambda_\rho \leq 1 \right\}\subset N'_\R
  $$
  are given by its vertices.
 \end{itemize}
\end{proposition}
From the discussion in Section \ref{sec:toric quotient}, the following proposition concludes the proof of Theorem \ref{theo:intro}.
\begin{proposition}
 \label{prop:Wterminal}
 The affine toric variety $W$ is Gorenstein and has terminal singularities.
\end{proposition}
\begin{proof}
We will exclude the rigid cases, and as in the proof of Proposition \ref{prop:Wtoric}, assume that $X$ is obtained from $Y_j$ by successive $C_p$-equivariant blow-ups, where $j\in\lbrace 2, 3, 4\rbrace$, for suitable $p\in\lbrace 2,3 \rbrace$.

 We use the characterisation in Proposition \ref{prop:Gorenstein and terminal}, and first need to describe the rays of $\sigma$ and their primitive generators. Note that by construction, the rays of $\sigma$ belong to the set $\lbrace \R_+\cdot A(\tilde e_i),\: 1\leq i\leq d \rbrace$. Let $\rho_i=\R_+\cdot A(\tilde e_i)$ be such a ray. We claim that $A(\tilde e_i)$ is primitive in $N'$. The argument is similar to the one used to prove strict convexity in the proof of Proposition \ref{prop:Wtoric}. Suppose by contradiction that there is $a\in \N$, $a\geq 2$, and $\tilde e \in \tilde \sigma\cap \tilde N$ such that $A(\tilde e_i)=a A(\tilde e)$. Notice that there is $x\in N$ such that 
 \begin{equation}
  \label{eq:Bjx}
  B(x)=(B_j(x),B_+(x))=\tilde e_i - a \tilde e.
 \end{equation}
 By injectivity of $B_j$, if $B_j(x)=0$, then $x=0$ hence $\tilde e_i = a\tilde e$. This is absurd as $\tilde e_i$ is primitive. So we may assume $B_j(x)\neq 0$. Similarly, using Equation (\ref{eq:Im Bj}), we may assume as well that $B_j(x)\notin - \tilde \sigma_j$ and $B_j(x)\notin \tilde \sigma_j$. Hence $B_j(x)$ must have at least one positive and one negative coordinate in the basis $(\tilde e_k)_{1\leq k \leq d_j}$. As $\tilde e \in \tilde \sigma\cap \tilde N$, its coordinates in the basis $(\tilde e_k)_{1\leq k\leq d}$ are non-negative integers, and as $a\geq 2$,  we can write 
 $$
 a\tilde e = \sum_{k=1}^d a_k \tilde e_k
 $$
 with $a_k=0$ or $a_k\geq 2$. Then, from Equation (\ref{eq:Bjx}), $B_j(x)$ has exactly one coordinate equal to $1$, while its other non-zero coordinates are all less or equal to $-2$. A case by case analysis using the description of $V_j$ in Equation (\ref{eq:example Vs}) (cf Example \ref{example:deformation spaces}) shows that this is impossible. Indeed, for $V_2$, the map $B_2$ is
 $$
 B_2(x_1,x_2)=(x_2,-x_2, -x_1+x_2, x_1-x_2),
 $$
 with $x=(x_1,x_2)\in\R^2=N_\R$,
 so if one coordinate of $B_2(x)$ is $1$, there is another coordinate that equals $-1$. A similar argument leads to the same conclusion for $B_4$. For $B_3$, we can simply use the fact that the weight spaces are $2$ dimensional, and we have
 $$
 B_3(x)=(-x_2,-x_2,x_1,x_1,-x_1+x_2,-x_1+x_2),
 $$
 so the value $1$ would appear at least twice in the coordinates of $B_3(x)$, if it ever does.
 
 We will then prove that there is  $m' \in M'$ such that $$ \forall i\in [\![1, d ]\!]\: ,\:\langle m', A(\tilde e_i) \rangle =1 ,$$
 which implies that $W$ is Gorenstein. So let $m'\in M'$. Set 
 $$\tilde m = A^*(m')\in \ker (B^*).$$
 Then,
 $$
 \begin{array}{ccc}
   \forall i\in [\![1, d ]\!]\:,\: \langle m', A(\tilde e_i) \rangle =1 & \Longleftrightarrow & \forall i\in [\![1, d ]\!]\:,\: \langle \tilde m, \tilde e_i \rangle =1 \\
    & & \\
    & \Longleftrightarrow & \tilde m = (1, \ldots  ,1)
 \end{array}
$$
where we used the basis $(\tilde e_i)$ to produce the coordinates of $\tilde m$.
Hence, it is equivalent to show that 
$$
(1, \ldots, 1)\in \ker (B^*),
$$
which by definition of $B$ is equivalent to 
\begin{equation*}
 \label{eq:weights vanishing condition}
\sum_{i=1}^d  m_i =0\in M
\end{equation*}
where $m_i$'s are the weights describing the $T$-action on $V$. This is where we use that $X$ is foldable. Recall that $G=\Aut(N,\Sigma)$. The $G$-action on $N$ naturally induces a $G$-action on $M$ by duality, in such a way that the duality pairing is $G$-invariant. Explicitely, for $m\in M=\Hom(N,\Z)$, and $g\in G$, we have
$$
g\cdot m := m(g^{-1}\,\cdot)
$$
and thus 
$$\langle g\cdot m , g\cdot u \rangle =\langle m , u\rangle$$ for any $u\in N$.
By the characterisation of the weights $(m_i)$'s in Equation (\ref{eq:iltenformula}), we see that this $G$-action preserves the set of weights appearing in the decomposition $V=\bigoplus_{m\in M} V_m$. As $G$ contains a non-trivial cyclic group, there is an element $g\in\Aut(N)$ of order $p$ with no fixed point but $0$. This element generates a cyclic group that acts freely on $\lbrace m_i,\: 1\leq i \leq d \rbrace$ (freeness comes from the fact that the action is free on the whole $M_\R$). Hence  
\begin{equation*}
 \label{eq:weights vanishing}
\sum_{i=1}^d  m_i=\sum_{i'} \sum_{k=0}^{p-1} g^k\cdot m_i'
\end{equation*}
where we picked a single element $m_i'$ in each orbit under this action. As $g\neq \Id$, we have 
$$
\Id + g + \cdots + g^{p-1}=0
$$
so that for each orbit
$$
\sum_{k=0}^{p-1} g^k\cdot m_i'=0.
$$
We then deduce the existence of the required $m'$, and that $W$ is Gorenstein.

We proceed to the proof of the fact that $W$ has terminal singularities. Let 
$$
  \Pi_\sigma:= \left\{ \sum_{\rho\in \sigma(1)} \lambda_\rho u_\rho\:\vert\: \sum_{\rho\in\sigma(1)} \lambda_\rho \leq 1,\: 0\leq \lambda_\rho \leq 1 \right\}.
  $$
  As described above, we may fix a subset of 
  $$\lbrace A(\tilde e_i),\: 1\leq i\leq d \rbrace$$ as a set of generators for the rays in $\sigma(1)$.
  Let $u'\in \Pi_\sigma\cap N'$ given by 
  $$
  u'=\sum_{i=1}^d \lambda_i A(\tilde e_i),
  $$
  with 
  $$0\leq \lambda_i \leq 1,$$
  and 
  $$
  \lambda_1 +\ldots +\lambda_d \leq 1,$$
  assuming $\lambda_i=0$ when $A(\tilde e_i)$ is not in our fixed chosen set of ray generators. Assume that there is $i_0$ with $\lambda_{i_0}\neq 0$. We need to show that $\lambda_{i_0}=1$ and $\lambda_i=0$ for $i\neq i_0$. By assumptions on the $(\lambda_i)$'s, it is enough to show that $\lambda_i\in\Z$ for all $i$. As $u'\in N'$, there exists $\tilde u\in \tilde N$ such that 
  $$
  A (\tilde u )= A(\sum_{i=1}^d \lambda_i \tilde e_i),
$$
and then there is $x\in N_\R$ such that
$$
\tilde u - \sum_{i=1}^d \lambda_i \tilde e_i=B(x)=(B_j(x), B_+(x)),
$$
where we recall that we still denote by $A$ and $B$ their $\R$-linear extensions to $\tilde N_\R$ and $N_\R$.
By injectivity of $B_j$ again, we only need to consider the case when $B_j(x)\neq 0$ and $1\leq i_0\leq d_j$. Indeed, if $B_j(x)=0$ then by injectivity $x=0$ and $\tilde u$ being in $\tilde N$ forces the $\lambda_i$'s to be integral, whereas if $\lambda_i=0$ for all $i\leq d_j$, then $B_j(x)\in \tilde N$ (recall that $B_j$ maps to the first $d_j$-coordinates) and so by saturation of $B_j(N)$ we have $x\in N$ and thus $B(x)\in \tilde N$ which together with $\tilde u\in\tilde N$ implies that the $\lambda_i$'s are integral. Hence, we reduced ourselves to study the case when $X=Y_j$. We will use again the explicit descriptions of the $V_j$'s from Example \ref{example:deformation spaces}. For $V_3$, we find the system 
$$
\left\{ 
\begin{array}{ccc}
 \tilde u_1 - \lambda_1 & = & -x_2\\
 \tilde u_2 - \lambda_2 & = & -x_2\\
 \tilde u_3 - \lambda_3 & = & x_1\\
 \tilde u_4 - \lambda_4 & = & x_1\\
 \tilde u_5 - \lambda_5 & = & -x_1+x_2\\
 \tilde u_6 - \lambda_6 & = & -x_1+x_2
\end{array}
\right. 
$$
where $x=(x_1,x_2)\in\R^2=N_\R$ and $\tilde u=(\tilde u_1,\ldots,\tilde u_6)\in \Z^6\simeq \tilde N$ (using the basis $(\tilde e_i)_{1\leq i \leq 6}$),
from which we obtain 
$$
( \lambda_1 + \lambda_3+\lambda_5, \lambda_1 + \lambda_3+\lambda_6,\lambda_1 + \lambda_4+\lambda_6,\lambda_1 + \lambda_4+\lambda_5)\in\N^4
$$
and 
$$
(\lambda_2 + \lambda_3+\lambda_5, \lambda_2 + \lambda_3+\lambda_6,\lambda_2 + \lambda_4+\lambda_6,\lambda_2 + \lambda_4+\lambda_5)\in\N^4,
$$
where we used that the $\tilde u_i$'s are integers and  $\lambda_i\geq 0$ for all $i$.
 Together with
\begin{equation}
 \label{eq:sumbetween}
0 \leq \lambda_1 + \ldots + \lambda_6 \leq 1,
\end{equation}
we deduce that actually
$$
( \lambda_1 + \lambda_3+\lambda_5, \lambda_1 + \lambda_3+\lambda_6,\lambda_1 + \lambda_4+\lambda_6,\lambda_1 + \lambda_4+\lambda_5)\in\lbrace 0,1\rbrace^4
$$
and 
$$
(\lambda_2 + \lambda_3+\lambda_5, \lambda_2 + \lambda_3+\lambda_6,\lambda_2 + \lambda_4+\lambda_6,\lambda_2 + \lambda_4+\lambda_5)\in\lbrace 0,1\rbrace^4.
$$
 Let's then assume that $i_0=1$ to fix the ideas, so $\lambda_1>0$, the other cases being similar. We then must have $\lambda_1 + \lambda_3+\lambda_5=1$, and summing with $\lambda_2 + \lambda_4+\lambda_6$, we see with (\ref{eq:sumbetween}) that $\lambda_2+\lambda_4+\lambda_6=0$. Hence $\lambda_2=\lambda_4=\lambda_6=0$. Then, $\lambda_1+\lambda_4+\lambda_6=\lambda_1>0$ so $\lambda_1=1$. Considering $\lambda_1 + \lambda_3+\lambda_6=1$ and $\lambda_1 + \lambda_4+\lambda_5=1$ yieds to $\lambda_3=\lambda_5=0$, and the result follows in that case. The cases $V_2$ and $V_4$ are similar, so we will only treat $V_4$. In that case, in the coordinates $\tilde x=(\tilde x_1,\ldots,\tilde x_4)\in\R^4\simeq \tilde N_\R$ given by $(\tilde e_1,\ldots,\tilde e_4)$, the map $A$ is given by
 $$
 A(\tilde x_1,\tilde x_2,\tilde x_3,\tilde x_4)=(\tilde x_1+\tilde x_2,\tilde x_3+\tilde x_4).
 $$
 Hence
$$
\left\{ 
\begin{array}{ccc}
 A(\tilde e_1) & = & A(\tilde e_2) \\
 A(\tilde e_3) & = & A(\tilde e_4)
\end{array}
\right. 
$$
so we can pick the ray generators $(A(\tilde e_1), A(\tilde e_3))$ for the two-dimensional cone $\sigma$. The description of $V_4$ then implies
$$
\left\{ 
\begin{array}{ccc}
 \tilde u_1 - \lambda_1 & = & x_1\\
 \tilde u_2  & = & -x_1\\
 \tilde u_3 - \lambda_3 & = & x_2\\
 \tilde u_4  & = & -x_2
\end{array}
\right. 
$$
hence $x\in N$ and the result follows.
  \end{proof}
  
  \subsection{Examples}
  \label{sec:examples}
  We first provide two examples of local moduli spaces and then an example of a quotient $H^1(X,TX)//T$ that is not $\Q$-Gorenstein, for $X$ a smooth toric surface.
  
  \subsubsection{A smooth example : $Y_4$} We consider the toric variety $W_4$ which is the quotient of $V_4$ by $T$ (recall the definition of $Y_4$ in Example \ref{example:deformation spaces}, and that $V_4=H^1(Y_4,TY_4)$). As seen in the proof of Proposition \ref{prop:Wterminal}, the generators of $\sigma$ in that case can be taken to be $A(\tilde e_1)$ and $A(\tilde e_3)$, so that $\sigma$ is isomorphic to the cone
  $$
  \R_+\cdot (1,0)+\R_+\cdot (0,1) \subset \R^2.
  $$
  Then, 
  $$
  W_4\simeq \C^2,
  $$
  and the $G$-action on $V_4$ (see Remark \ref{rem:Gaction}), descends to a $D_1$-action on $W_4\simeq \C^2$ generated by a reflection
  $$
  (x,y)\mapsto (y,x).
  $$
  Hence, we conclude that the local moduli space is modeled on 
  $$
  W_4//G \simeq \C^2.
  $$

  \subsubsection{A singular example : $Y_3$}
  Let's consider now $W_3$, given by the GIT quotient of $V_3$ by $T$. We can pick isomorphisms 
  $
  N\simeq \Z^2,
  $
  $
  \tilde N\simeq \Z^6
  $
  and 
  $$
  N'= \tilde N/N\simeq \Z^4
  $$
  such that the map $B : N\to \tilde N$ is given by 
  $$
  B=\left[ 
  \begin{array}{cc}
   0 & -1 \\
   0 & -1 \\
   1 & 0 \\
   1 & 0 \\
   -1 & 1 \\
   -1 & 1 
  \end{array}
  \right]
  $$
  and the map $A : \tilde N \to N'$ by 
  $$
   A=\left[ 
  \begin{array}{cccccc}
   -1 & 1 & 0 & 0 & 0 & 0 \\
  0 & 0 & -1 & 1 & 0 & 0 \\
   1 & 0 & 1 & 0 & 1 & 0 \\
  1 & 0 & 1 & 0 & 0 & 1 
  \end{array}
  \right].
  $$
  Hence, $W_3$ is the toric variety associated to the cone
  $$
  \sigma = \sum_{i=1}^6 \R_+\cdot  e_i'\subset \R^4
  $$
  where $(e_i')_{1\leq i\leq 6}$ are given by the columns of the matrix $A$. 
  This toric affine variety is singular. It is not even simplicial, as the $(e_i')_{1\leq i\leq 6}$, which are the primitive generators of the six rays in $\sigma(1)$, do not form a $\Q$-basis for $N'_\Q$. The $G$-action descends to a $D_3$-action on $W_3$. Indeed, as explained in the proof of Proposition \ref{prop:Wterminal}, $G\simeq D_3$ acts on the set of weights $(m_i)$'s that appear in the weight space decomposition of $V_3\simeq \C^6$. This action results in permutations of the $\chi^{m_i}$'s, and thus in permutations of the elements of the basis $(\tilde e_i)_{1\leq i\leq 6}$ (we use the notations from the previous proofs of Propositions \ref{prop:Wtoric} and \ref{prop:Wterminal}). Then, this $D_3$-action descends to $N'\simeq \Z^4$, by permutations of the $e_i'$'s. Explicitely, we compute that the  associated representation 
  $$
  D_3\to \mathrm{GL}_4(\Z),
$$ 
is generated by 
$$
\left[ 
  \begin{array}{cccc}
   0 & 0 & -1 & 1 \\
  1 & 0 & 0 & 0 \\
   0 & 0 & 1 & 0  \\
  0 & 1 & 1 & 0 
  \end{array}
  \right] \mathrm{ and } 
  \left[ 
  \begin{array}{cccc}
   -1 & 0 & 0 & 0 \\
  0 & -1 & 0 & 0 \\
   1 & 1 & 0 & 1  \\
  1 & 1 & 1 & 0 
  \end{array}
  \right]
$$
where the first matrix represents an element of order $3$ and the second one a reflection. The final quotient 
$W_3//D_3$ provides a singular example of local moduli space.

  \subsection{A non--Gorenstein example}
We produce here  an example of a toric surface $X$ with 
$\Aut(X)\simeq T$, and $H^1(X,TX)//T$ a non $\Q$-Gorenstein toric variety (see Figure 13). For this, simply blow-up $Y_4$ in a single point to produce the following fan  :
\begin{figure}[htbp]
\label{figXlast}
\psset{unit=0.95cm}
\begin{pspicture}(-3.5,-4.5)(3.5,1)
$$
\xymatrix @M=0mm{
\bullet & \bullet & \bullet & \bullet & \bullet \\
\bullet & \bullet & \bullet & \bullet & \bullet \\
 \bullet & \bullet & \bullet\ar[u]\ar[r]\ar[l]\ar[d]\ar[ru]\ar[ld]\ar[lu]\ar[rd]\ar[ruu]& \bullet & \bullet \\
\bullet & \bullet & \bullet & \bullet & \bullet \\
\bullet & \bullet & \bullet & \bullet & \bullet 
}
$$
\end{pspicture}
\caption{Fan of $X$.}
\end{figure}
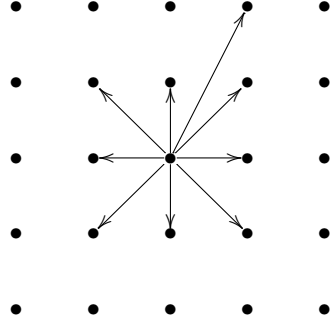
  
  As $X$ is blown-up from $Y_4$, we have $\Aut^0(X)\simeq T$. One can also check that we actually have $\Aut(X)\simeq T$. Moreover, the arguments in Section \ref{sec:toric quotient} go through, and we find that if $X$ were cscK, then its local moduli space would be modeled on the GIT quotient of $H^1(X,TX)$ by $T$. A direct computation again, using the method in Section \ref{sec:deformations of toric surfaces}, provides the weight space decomposition
  $$
  H^1(X,TX)= V^1_{+e_1^*}\oplus V^2_{-e_1^*}\oplus V^1_{+e_2^*}\oplus V^1_{-e_2^*}\oplus V^1_{e_1^*-e_2^*}.
  $$
  Given that the sum of the weights that appear in this decomposition doesn't vanish, the quotient $H^1(X,TX)//T$ is not $\Q$-Gorenstein (see the proof of Proposition \ref{prop:Gorenstein and terminal}).
  
  \section{Discussion and perspectives}
  \label{sec:discuss}
 
 \subsection{Relations between local moduli and Weil--Petersson metrics}
 \label{sec:relate moduli}
 We will discuss in this section the fact that the local moduli spaces we considered are related by toric fibrations. Assume that $\pi : X \to X_0$ is a $G$-equivariant blow-up between two foldable toric surfaces. We keep the notations from previous sections, using the subscript $0$ to refer to the spaces associated to $X_0$. From \cite[Corollary 1.6]{Ilten11}, the corresponding space $V_0$ injects in $V$. It is a straightforward exercise to check that we have the following commutative diagram : 
$$
\begin{array}{ccccccccc}
  0 &\longrightarrow & N &\stackrel{B}{\longrightarrow} &\tilde N &\stackrel{A}{\longrightarrow} & N'& \longrightarrow & 0 \\
   & & \downarrow & &  \downarrow & &  \downarrow & & \\
  0 & \longrightarrow & N  & \stackrel{B_0}{\longrightarrow} & \tilde N_0 &  \stackrel{A_0}{\longrightarrow} & N_0' & \longrightarrow & 0
\end{array}
$$
where the first vertical arrow is the identity and the last two are surjective. By construction, one sees that the surjective map $N'\to N_0'$ is compatible with $\sigma$ and $\sigma_0$, and induces a toric locally trivial fibration $W\to W_0$ whose fiber is itself an affine toric variety. The whole construction is $G$-equivariant, and provides maps between the associated local moduli spaces (up to shrinking the neighborhoods we considered). 

The cscK metric on $X$ lives in the class $\pi^*[\om_0]-\ep E_i $, where $\om_0$ is cscK on $X_0$, $\ep$ is small and the $E_i$'s stand for the exceptional divisors of the blow-up. It would be interesting to understand the behaviour of the associated Weil--Petersson metrics $(\Om^{WP}_{\ep})_{0 < \ep < \ep_0}$ on the local moduli spaces $\cW_\ep\subset W/G$  as constructed in \cite{DN} when $\ep$ goes to zero. It seems natural to  expect that the volume of the fibers of the fibration $W/G\to W_0/G$ go to zero, so that $\cW_\ep$ would converge in Gromov--Hausdorff sense to $\cW_0$.

\subsection{Higher dimensional case}
\label{sec:higher dimension}
Many features that hold for toric surfaces fail in higher dimension. First, even when it is reductive, the identity component of the automorphism group of a non-rigid toric variety will not a priori be isomorphic to the torus (see \cite{Nill}). Then, from dimension $3$, toric varieties may be obstructed (see \cite{IltenTuro}). Finally, the toric MMP (that produces the classification of toric surfaces) produces singular varieties in general. So our main ingredients to prove Theorem \ref{theo:intro} do not generalise, a priori, in higher dimension.

Nevertheless, it would be interesting to study what goes through the new difficulties. We expect  the right extension of the notion of foldable fans in higher dimension to be fans whose lattice automorphism group admit a subgroup that acts with no fixed point. This raises several questions :
\begin{enumerate}
 \item Do all crystallographic groups arise as lattice automorphism groups of smooth complete fans?
 \item Do all foldable toric varieties admit a cscK metric?
 \item Are foldable toric varieties unobstructed?
 \item What are the singularities of the moduli space of cscK metrics around cscK foldable toric varieties?
\end{enumerate}
 
\bibliographystyle{amsplain}
\bibliography{RepDefAbelian}

\end{document}